\documentclass[11pt]{amsart}
\usepackage{fullpage}
\usepackage{color}
\usepackage{pstricks,pst-node,pst-plot} 
\usepackage{graphicx,psfrag} 
\usepackage{color} 
\usepackage{tikz}
\usepackage{pgffor}
\usepackage{hyperref}
\usepackage{todonotes}
\usepackage{subfigure}

\usepackage{bm}

\usepackage{perpage}	 
	
\MakePerPage{footnote}	 

\newtheorem{theorem}{Theorem}[section]
\newtheorem*{theorem-non}{Theorem}
\newtheorem{lemma}[theorem]{Lemma}
\newtheorem{corollary}[theorem]{Corollary}
\newtheorem{proposition}[theorem]{Proposition}

\newtheorem{assumption}[theorem]{Assumption}

\theoremstyle{definition}
\newtheorem{definition}[theorem]{Definition}
\newtheorem{example}[theorem]{Example}

\theoremstyle{remark}
\newtheorem{remark}[theorem]{Remark}

\numberwithin{equation}{section}

\definecolor{gray}{rgb}{.5,.5,.5}

\definecolor{black}{rgb}{0,0,0}

\definecolor{blue}{rgb}{0,0,1}

\definecolor{red}{rgb}{1,0,0}

\definecolor{green}{rgb}{0,1,0}

\definecolor{yellow}{rgb}{1,1,.4}

\newrgbcolor{purple}{.5 0 .5}
\newrgbcolor{black}{0 0 0}
\newrgbcolor{white}{1 1 1}
\newrgbcolor{gold}{.5 .5 .2}
\newrgbcolor{darkgreen}{0 .5 0}
\newrgbcolor{gray}{.5 .5 .5}
\newrgbcolor{lightgray}{.75 .75 .75}
\newrgbcolor{lightred}{.75 0 0}

\newcommand{\Plus}{\mathord{\begin{tikzpicture}[baseline=0ex, line width=1, scale=0.13]
\draw (1,0) -- (1,2);
\draw (0,1) -- (2,1);
\end{tikzpicture}}}

\newcommand{\Minus}{\mathord{\begin{tikzpicture}[baseline=0ex, line width=1, scale=0.13]
\draw (0,1) -- (2,1);
\end{tikzpicture}}}

\newcommand{\crossneg}{
\begin{tikzpicture}[baseline=-2]
\draw[white,line width=1.5pt,double=black,double distance=.5pt] (0,-0.1) -- (0.3,0.2);
\draw[white,line width=1.5pt,double=black,double distance=.5pt] (0,0.2) -- (0.3,-0.1);
\end{tikzpicture}}

\newcommand{\crosspos}{
\begin{tikzpicture}[baseline=-2]
\draw[white,line width=1.5pt,double=black,double distance=.5pt] (0,0.2) -- (0.3,-0.1);
\draw[white,line width=1.5pt,double=black,double distance=.5pt] (0,-0.1) -- (0.3,0.2);
\end{tikzpicture}}

\newcommand{\etalchar}[1]{$^{#1}$}

\begin{document}

\title{A lower bound on the average genus of a 2-bridge knot}

\author{Moshe Cohen}
\address{Mathematics Department, State University of New York at New Paltz, New Paltz, NY 12561}
\email{cohenm@newpaltz.edu}

\begin{abstract}
Experimental data from Dunfield et al using random grid diagrams suggests that the genus of a knot grows linearly with respect to the crossing number.  Using billiard table diagrams of Chebyshev knots developed by Koseleff and Pecker and a random model of 2-bridge knots via these diagrams developed by the author with Krishnan and then with Even-Zohar and Krishnan, we introduce a further-truncated model of all 2-bridge knots of a given crossing number, almost all counted twice.  We present a convenient way to count Seifert circles in this model and use this to compute a lower bound for the average Seifert genus of a 2-bridge knot of a given crossing number.
\end{abstract}


\maketitle

\section{Introduction}

\subsection*{On randomness}  In recent years, many topologists and geometers have utilized randomness to study the behaviors of their favorite mathematical objects.  However, any random model has inherent biases.  Erlandsson, Souto, and Tao \cite{EST:mcg} show that the set of pseudo-Anosov elements of the mapping class group of a surface is generic with respect to several metrics, confirming a result by Maher \cite{Mah} using a metric given by random walks on the Cayley graph of the mapping class group.  On the other hand, Malyutin \cite{Mal} and with Belousov \cite{BelMal} (and see also \cite{Mal:q}) show that hyperbolic links and knots are not generic, that the proportion of satellite knots among all prime non-split links does not converge to zero, contradicting results by Ma \cite{Ma} and Ito \cite{Ito} showing that the closure of a random braid is a hyperbolic link and also a result by Ichihara and Ma \cite{IcMa} showing that a random link via bridge position is hyperbolic.

A solution is to study these objects from \emph{many} different random models with hopes of determining whether these properties are inherent or intrinsic to the model alone.  Dunfield and Thurston \cite{DT:fin} list several models for random three-manifolds that could be developed further.  Work by Petri with Baik, Bauer, Gekhtman, Hamenst\"adt, Hensel, Kastenholz, and Valenzuela \cite{BBGHHKPV}, Thaele \cite{PetTha}, Mirzakhani \cite{MirPet}, and Raimbault \cite{PetRai} showcases the benefits of exploring various models.

Another important avenue is to study further properties of the mathematical objects appearing in previous random models.  Manin \cite{Man:rand} finds bounds for the areas of minimal Seifert surfaces of knot obtained from random embeddings of a polygon studied by Millet \cite{Mil:MC} via Monte Carlo exploration.

The genus of an object is often a particularly accessible invariant for calculations on randomness.  Linial and Nowik \cite{LinNow} draw a correspondence between generic curves in oriented surfaces and oriented chord diagrams and then show that a randomly chosen oriented chord diagram of order $n$ has expected genus $\frac{n}{2}-\Theta(\ln n)$.  Chmutov and Pittel \cite{ChmPit:genus} draw a correspondence between surfaces obtained by gluing the sides of an $n$-gon and chord diagrams and then show that the distribution of the genus of a random chord diagram is
asymptotically Gaussian.

Brooks and Makover \cite{BroMak} obtain a random Riemann surface by randomly gluing together an even number of triangles and then show that the expected genus is between $(1 + \frac{n}{2}) \pm \Theta(\log n)$.  Chmutov and Pittel \cite{ChmPit:poly} randomly glue together the sides of $n$ oriented polygonal discs and then show that with high probability the surface consists of a single component and its genus is asymptotic to a Gaussian random variable.  Even-Zohar and Farber \cite{EZFar} randomly glue together \emph{some} of the sides from the two models above and then show that the genus and number of boundary components asymptotically follow a bivariate normal distribution.

Shrestha \cite{Shr:genus} uses the symmetric group to study random square-tiled surfaces and shows that the genus satisfies a local central limit theorem.


\subsection*{On random knot diagrams} 
How does a given numerical knot invariant grow with respect to the crossing number of the knot?  Random knots can be used to answer this question, provided the model for random knots is conducive to calculating the invariant.  Since several knot invariants are calculated diagrammatically, it will be useful to find models for random knot \emph{diagrams}, particularly diagrams that have been well-studied.

Grid diagrams, introduced (and re-introduced) in \cite{Crom:grid, Brun, Dyn}, were popularized through connections to knot Floer homology in \cite{MOST, MOS}.  Dunfield et al \cite{Dun:knots} developed a random model for these grid diagrams, and Doig \cite{Doig} studied the number of components.   
These grid diagrams can be re-envisioned as petal (or Petaluma) knot diagrams with ``\"ubercrossings'' introduced by Adams et al \cite{Adams:petal, Adams:petal2, Adams:petal3} and studied by Colton et al \cite{CGHS}.  Even-Zohar, Hass, Linial, and Nowik \cite{EZHLN, EZHLN:distpetal} developed a random model for these petal knot diagrams.  It is more difficult to sample \emph{all} knot diagrams, as in work by Chapman \cite{chapman2017asymptotic} and with Cantarella and Mastin \cite{CanChaMas}; using a particular diagrammatic model, for example straight knots studied by Owad \cite{Owad:straight}, may lead to easier computations.

Chebyshev diagrams (of long knots) were introduced by Fischer \cite{Fisc} and studied by Koseleff and Pecker \cite{KosPec3, KosPec4} together with others \cite{KosPecRou:10crossings, KosPecTran:diagrams, BruKosPec} and by this author \cite{Co:clock, Co:3bridge}.  These are related to Lissajous knot diagrams studied by the late Vaughn Jones (1952-2020) with Bogle, Hearst, and Stoilov \cite{BHJS} and Przytycki \cite{JonPrz}.

A Chebyshev diagram can be obtained from the projection of a curve parametrized in three dimensions $x=T_a(t)$, $y=T_b(t)$, and $z=T_c(t+\phi)$ by Chebyshev polynomials $T_n(t)$ for some $a,b,c,n\in\mathbb{N}$ and a phase shift $\phi\in\mathbb{R}$.  Projections of these Chebyshev diagrams can be re-envisioned as trajectories on $a\times b$ billiard tables, with $a$ and $b$ relatively prime for knots, as follows:  
a billiard ball fired at $45^\circ$ from the lower left corner bounces through every $1\times1$ square with slope either 1 or -1 and leaves the table at one of the corners on the right as in Figure \ref{fig:Cheb}; including crossing information gives \emph{billiard table diagrams} $\widetilde{D}$.

\renewcommand{\thesubfigure}{}
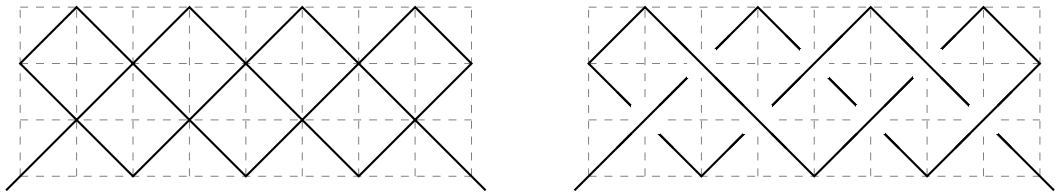
\begin{figure}[h]
\centering
\subfigure{
\begin{tikzpicture}[scale=.75]
    \foreach \i in {0,...,8} {
        \draw [very thin,dashed,gray] (\i,1) -- (\i,4);
    }
    \foreach \i in {1,...,4} {
        \draw [very thin,dashed,gray] (0,\i) -- (8,\i);
    }
\draw[thick](-.25,.75) -- (3,4) -- (6,1) -- (8,3) -- (7,4) -- (4,1) -- (1,4) -- (0,3) -- (2,1) -- (5,4) -- (8.25,.75);
\end{tikzpicture}
}
\hspace{2em}
\subfigure{
\begin{tikzpicture}[scale=.75]
    \foreach \i in {0,...,8} {
        \draw [very thin,dashed,gray] (\i,1) -- (\i,4);
    }
    \foreach \i in {1,...,4} {
        \draw [very thin,dashed,gray] (0,\i) -- (8,\i);
    }
\draw[thick](-.25,.75) -- (3,4) -- (6,1) -- (8,3) -- (7,4) -- (4,1) -- (1,4) -- (0,3) -- (2,1) -- (5,4) -- (8.25,.75);
    \foreach \i/\j in {1/2,4/3,5/2,7/2} {
        \draw[thick, black] (\i-.5,\j+.5) -- (\i+.5,\j-.5);
				\fill[white] (\i-.25,\j-.25) rectangle (\i+.25,\j+.25);
        \draw[thick, black] (\i-.5,\j-.5) -- (\i+.5,\j+.5);
    }
    \foreach \i/\j in {2/3,3/2,6/3} {
        \draw[thick, black] (\i-.5,\j-.5) -- (\i+.5,\j+.5);
				\fill[white] (\i-.25,\j-.25) rectangle (\i+.25,\j+.25);
        \draw[thick, black] (\i-.5,\j+.5) -- (\i+.5,\j-.5);
    }
\end{tikzpicture}
}
\caption{\label{fig:Cheb} The projection of a $3\times 8$ billiard table diagram and a (reduced) billiard table diagram $\widetilde{D}$ obtained from the (reduced) billiard table word $w=+-{}-++-+$.}
\end{figure}

In particular Koseleff and Pecker show that all knots can be realized by Chebyshev diagrams.  Knots with bridge number $br$ can be realized for  any $a\geq 2br-1$ and $a<b$.

The author with Krishnan \cite{CoKr} and with Even-Zohar and Krishnan \cite{CoEZKr} developed a random model for these $a=3$ Chebyshev diagrams by taking a random string of $\{+,-\}^n$, where $n=b-1$ is the number of crossings of the diagram and where $+$ and $-$ correspond to the slope of the overstrand in the crossing:  $\crosspos$ and $\crossneg$, respectively.  Because $a=3$, the only knots appearing in this model are 2-bridge knots and the unknot.  The sequel paper computed an exact formula for the probability of a such a knot with crossing number $c$ appearing in a random binary string of length $n$.

One reason to restrict to the setting of 2-bridge knots (or rational knots) is that they can easily be described either by (finite) continued fractions (giving rational numbers), which we use below, or by closures of rational tangles, as developed by the late John Conway (1937-2020) \cite{Con}.

\subsection*{On the Seifert genus of a knot} 
Kanenobu \cite[Assertion 2]{Kan:genus} gives a formula for the Seifert genus of a 2-bridge knot based on the notation for Conway's normal form.  
Murasugi \cite{Mur:genus} shows that for any alternating knot with a constant incidence number, the genus is exactly equal to one half of the degree of its Alexander polynomial.  Crowell \cite[Theorem 3.5]{Cro:genus} shows that for an alternating link type with multiplicity, the degree of its reduced Alexander polynomial is one less than twice the genus plus the multiplicity.  Suzuki and Tran \cite{SuzTra} relate genera of 2-bridge knots to epimorphisms of the respective knot groups.

Baader, Kjuchukova, Lewark, Misev, and Ray \cite{BKLMR} show that the expected value of the ratio of the 4-genus $g_4$ and the Seifert genus $g$ over all 2-bridge knots whose 4-plat presentation, in the sense of Conway, has $2n$ crossings is $\lim_{n\rightarrow\infty}\left\langle\frac{g_4}{g}\right\rangle_n=0$.

Dunfield et al \cite{Dun:knots} use rejection sampling to compile a list of one knot per crossing number for each crossing number up to 1000 and compute upper and lower bounds on the genus of each of these knots.  Their experimental data suggests that the genus of any knot grows linearly with respect to its crossing number.  
  This result especially motivated the work in the present paper.


\subsection*{On the present paper} 
 Here only 2-bridge knots will be considered, and so $c\geq 3$.  We consider \emph{all} 2-bridge knots of crossing number $c$, counted exactly once or twice according to Assumption \ref{ass:model} and Remark \ref{rem:palindrome}, by considering every case of $\varepsilon_i\in\{1,2\}$ for $2\leq i \leq c-1$ and $\varepsilon_1=1=\varepsilon_c$ in Equations (\ref{eq:oddword}) and (\ref{eq:evenword}):
$$(+)^{\varepsilon_1}(-)^{\varepsilon_2}(+)^{\varepsilon_3}(-)^{\varepsilon_4}\ldots(-)^{\varepsilon_{c-1}}(+)^{\varepsilon_c} \text{ for $c$ odd and}$$
$$(+)^{\varepsilon_1}(-)^{\varepsilon_2}(+)^{\varepsilon_3}(-)^{\varepsilon_4}\ldots(+)^{\varepsilon_{c-1}}(-)^{\varepsilon_c} \text{ for $c$ even},$$
with reduced length $\ell=\sum_{i=1}^c \varepsilon_i \equiv 1 \mod 3$.  Each of these reduced billiard table words $w$ produces an alternating diagram $D$ by Theorem \ref{thm:alt}.  The total number of these words or diagrams produced is given by Theorem \ref{thm:number}.  See Section \ref{sec:model} for details.

From this alternating diagram $D$, the Seifert genus $g(K)=1-\frac{1+s-c}{2}=\frac{1-s+c}{2}$ of the knot $K$ can be computed via Seifert's algorithm \cite{Sei} counting the number of Seifert circles $s$ and using the crossing number $c=c(K)$.  We count this number via vertically-smoothed crossings in Theorem \ref{thm:Seifert} with bounds given in Corollary \ref{cor:upper}.  Examples \ref{ex:c6} and \ref{ex:c7} are worked out for $c=6$ and $c=7$, respectively, with figures inside Tables \ref{tab:c6} and \ref{tab:c7}, respectively.  See Section \ref{sec:Seifert} for details.

From this we arrive at the Main Theorem \ref{thm:main}, presented in a condensed form here.

\begin{theorem-non}
A lower bound on the average genus of a 2-bridge knot with given crossing number $c$ is roughly
\begin{equation*}
\frac{c-1}{2}-\left(\frac{3}{2(2^{c-2})}\right)\left(\sum_{i=2}^{c-1}\sum_{d_1=0}^{i-2}\sum_{\substack{d_2=0 \\ c+d_1+d_2+(0 or 1)\equiv1}}^{c-i-1}\binom{i-2}{d_1}\delta(i)\binom{c-i-1}{d_2} \right).
\end{equation*}

This average is taken over all 2-bridge knots appearing twice except for those with palindromic type only appearing once.
\end{theorem-non}

Corollary \ref{cor:sym} shortens this summation due to symmetry.  Examples \ref{ex:c6ii} and \ref{ex:c7ii} apply this Main Theorem to our $c=6$ and $c=7$ cases, respectively, showing the calculations explicitly.

\subsection*{Acknowledgements} The author would like to thank Nathan Dunfield for his experimental work on genus, John McCleary for his pointing out an important result, Marina Ville for her hospitality during which a portion of this work was completed, Christopher Panna and Lumina Resnick for their hospitality during which another portion of this work was completed, and Adam Lowrance for his reading of an earlier draft.

\section{Alternating diagrams of 2-bridge knots}
\label{sec:model}

Both \cite{CoKr} and \cite{CoEZKr} discuss moves on billiard table diagrams that are similar to Reidemeister moves:  the \emph{internal reduction move} that deletes a run $+++$ or $-{}-{}-$ of three in a row and the \emph{external reduction move} that deletes $++-$ or $--+$ only from the start of the word or $-++$ or $+-{}-$ only from the end of the word.

For crossing number $c\geq 3$, a word in $\{+,-\}$ can thus be reduced by these moves to one of the form
\begin{equation}
\label{eq:oddword}
(+)^{\varepsilon_1}(-)^{\varepsilon_2}(+)^{\varepsilon_3}(-)^{\varepsilon_4}\ldots(-)^{\varepsilon_{c-1}}(+)^{\varepsilon_c} \text{ for $c$ odd},
\end{equation}
\begin{equation}
\label{eq:evenword}
(+)^{\varepsilon_1}(-)^{\varepsilon_2}(+)^{\varepsilon_3}(-)^{\varepsilon_4}\ldots(+)^{\varepsilon_{c-1}}(-)^{\varepsilon_c} \text{ for $c$ even},
\end{equation}
$$(-)^{\varepsilon_1}(+)^{\varepsilon_2}(-)^{\varepsilon_3}(+)^{\varepsilon_4}\ldots(+)^{\varepsilon_{c-1}}(-)^{\varepsilon_c}  \text{ for $c$ odd, or}$$
$$(-)^{\varepsilon_1}(+)^{\varepsilon_2}(-)^{\varepsilon_3}(+)^{\varepsilon_4}\ldots(-)^{\varepsilon_{c-1}}(+)^{\varepsilon_c} \text{ for $c$ even}.$$
where $\varepsilon_1=1=\varepsilon_c$ and all other $\varepsilon_i\in\{1,2\}$.  These shall be referred to as \emph{reduced billiard table words}.

We refer to a run of three in a row as a \emph{triple}, a run of two as a \emph{double}, and a run of one as a \emph{single}.

By the following lemma, the number of ways of writing a 2-bridge knot in a reduced word is either 4 or 8:  
\begin{lemma}
\label{lem:numberreducedwords}
(adapted from Cohen-Krishnan \cite[Lemma 2.20]{CoKr} based on work by Schubert \cite{Sch} and Koseleff-Pecker \cite{KosPec4})
Let $K$ be a 2-bridge knot.  Then there are exactly two reduced lengths $\ell$ for $K$. Furthermore, these lengths are $\ell_0\equiv 0\text{ mod }3$ and $\ell_1\equiv 1\text{ mod }3$, and the number of ways $r(\ell)$ to write $K$ in reduced length $\ell$ is either $2$ or $4$.
\end{lemma}

From a reduced word $w$, another word yielding the same knot can be obtained by three techniques:  reversing the word; taking the mirror image of the word, replacing each $+$ with $-$ and vice versa; and replacing all interior $\varepsilon_i=1$ with $\varepsilon_i=2$ and vice versa for every $i=2,\ldots,c-1$.

\begin{assumption}
\label{ass:model}
To reduce the number of ways of writing a 2-bridge knot in a reduced word further to 1 or 2, only words beginning with $+$ and with length $\ell_1\equiv 1\text{ mod }3$ will be considered in the current work.  That is, we only consider reduced billiard table words of the forms in Equations (\ref{eq:oddword}) and (\ref{eq:evenword}).
\end{assumption}

\begin{remark}
\label{rem:palindrome}
All 2-bridge knots appear twice in this model except for \emph{palindromic} words with $c$ odd and words that are equal to the reverse of the mirror image of $w$ with $c$ even.  We will collectively refer to all of these words as having \emph{palindromic type}; these appear exactly once in the model.  
\end{remark}

Notationally it is convenient to view these knot diagrams $\widetilde{D}$ as braids:  starting from left to right, with the labeling of strands going upwards, but with a plat closure.  They are of the form
$$\sigma_1\sigma_2^{\pm1}\sigma_1^{\pm1}\sigma_2^{\pm1}\ldots\sigma_1 \text{ for $c$ odd or}$$
$$\sigma_1\sigma_2^{\pm1}\sigma_1^{\pm1}\sigma_2^{\pm1}\ldots\sigma_2^{-1} \text{ for $c$ even}.$$
We take the convention that the $i$th strand in $\sigma_i$ passes over the $i+1$st so that $\sigma_i^{+1}$ corresponds with a $+$ and $\sigma_i^{-1}$ corresponds with a $-$ in the reduced billiard table word $w$.

Reduced billiard table diagrams $\widetilde{D}$ from reduced billiard table words $w$ can be simplified to alternating diagrams $D$:

\begin{theorem}
\label{thm:alt}
(Mentioned without proof in \cite{CoEZKr}) A reduced billiard table diagram $\widetilde{D}$ from a reduced billiard table word $w$ corresponds to a unique alternating diagrams $D$.  
\end{theorem}

\begin{proof}
Figure \ref{fig:two-to-one} (\cite[similar to Figure 3]{CoEZKr}) demonstrates that every run of length two in $w$ corresponds to a single crossing in a new diagram $D$.

\begin{figure}[h]
\begin{center}
\begin{tikzpicture}[scale=0.5,white,double distance=2.0pt]
\draw[black] (0,0) -- (.5,.5);
\draw[black] (.5,1.5) -- (0,2) -- (1,3) -- (1.5,2.5);
\draw[black] (7.5,1.5) -- (8,2) -- (7,3) -- (6.5,2.5);
\draw[black] (8,0) -- (7.5,.5);
\draw[black] (11,0) -- (11.5,.5);
\draw[black] (11.5,1.5) -- (11,2) -- (12,3) -- (12.5,2.5);
\draw[black] (7.5,1.5) -- (8,2) -- (7,3) -- (6.5,2.5);
\draw[black] (8,0) -- (7.5,.5);
\draw[black] (16.5,.5) -- (17,0) -- (18,1) -- (17.5,1.5);
\draw[black] (17.5,2.5) -- (18,3);
    \foreach \i in {5,16} {
        \draw[black] (\i-.5,2.5) -- (\i,3) -- (\i+.5,2.5);
    }
    \foreach \i in {2,6,13,15} {
        \draw[black] (\i-.5,.5) -- (\i,0) -- (\i+.5,.5);
    }
    \foreach \i in {3} {
        \draw[thick, red] (\i-.5,2.5) -- (\i,3) -- (\i+.5,2.5);
    }
    \foreach \i in {4} {
        \draw[thick, red] (\i-.5,.5) -- (\i,0) -- (\i+.5,.5);
    }
\draw[->,black] (9,1.5) -- (10,1.5);
    \foreach \i/\j in {1/1,12/1} {
        \draw[black] (\i-.5,\j+.5) -- (\i+.5,\j-.5);
				\fill[white] (\i-.25,\j-.25) rectangle (\i+.25,\j+.25);
        \draw[black] (\i-.5,\j-.5) -- (\i+.5,\j+.5);
    }
    \foreach \i/\j in {3/1,4/2} {
        \draw[thick, red] (\i-.5,\j+.5) -- (\i+.5,\j-.5);
				\fill[white] (\i-.25,\j-.25) rectangle (\i+.25,\j+.25);
        \draw[thick, red] (\i-.5,\j-.5) -- (\i+.5,\j+.5);
    }
    \foreach \i/\j in {2/2,13/2} {
        \draw[black] (\i-.5,\j-.5) -- (\i+.5,\j+.5);
				\fill[white] (\i-.25,\j-.25) rectangle (\i+.25,\j+.25);
        \draw[black] (\i-.5,\j+.5) -- (\i+.5,\j-.5);
    }
    \foreach \i/\j in {14/2} {
        \draw[thick, red] (\i-.5,\j-.5) -- (\i+.5,\j+.5);
				\fill[white] (\i-.25,\j-.25) rectangle (\i+.25,\j+.25);
        \draw[thick, red] (\i-.5,\j+.5) -- (\i+.5,\j-.5);
    }
    \foreach \i/\j in {5/1,6/2,7/1,15/2,16/1,17/2} {
        \draw[black] (\i-.5,\j-.5) -- (\i+.5,\j+.5);
        \draw[black] (\i-.5,\j+.5) -- (\i+.5,\j-.5);
    }
        \draw[red] (13.5,.5) -- (14.5,.5);
\end{tikzpicture}\\
\begin{tikzpicture}[scale=0.5,white,double distance=2.0pt]
\draw[black] (0,0) -- (.5,.5);
\draw[black] (.5,1.5) -- (0,2) -- (1,3) -- (1.5,2.5);
\draw[black] (7.5,1.5) -- (8,2) -- (7,3) -- (6.5,2.5);
\draw[black] (8,0) -- (7.5,.5);
\draw[black] (11,0) -- (11.5,.5);
\draw[black] (11.5,1.5) -- (11,2) -- (12,3) -- (12.5,2.5);
\draw[black] (7.5,1.5) -- (8,2) -- (7,3) -- (6.5,2.5);
\draw[black] (8,0) -- (7.5,.5);
\draw[black] (16.5,.5) -- (17,0) -- (18,1) -- (17.5,1.5);
\draw[black] (17.5,2.5) -- (18,3);
    \foreach \i in {5,14,16} {
        \draw[black] (\i-.5,2.5) -- (\i,3) -- (\i+.5,2.5);
    }
    \foreach \i in {2,6,13} {
        \draw[black] (\i-.5,.5) -- (\i,0) -- (\i+.5,.5);
    }
    \foreach \i in {3} {
        \draw[thick, red] (\i-.5,2.5) -- (\i,3) -- (\i+.5,2.5);
    }
    \foreach \i in {4} {
        \draw[thick, red] (\i-.5,.5) -- (\i,0) -- (\i+.5,.5);
    }
\draw[->,black] (9,1.5) -- (10,1.5);
    \foreach \i/\j in {1/1,3/1,12/1,14/1} {
        \draw[black] (\i-.5,\j+.5) -- (\i+.5,\j-.5);
				\fill[white] (\i-.25,\j-.25) rectangle (\i+.25,\j+.25);
        \draw[black] (\i-.5,\j-.5) -- (\i+.5,\j+.5);
    }
    \foreach \i/\j in {15/1} {
        \draw[thick, red] (\i-.5,\j+.5) -- (\i+.5,\j-.5);
				\fill[white] (\i-.25,\j-.25) rectangle (\i+.25,\j+.25);
        \draw[thick, red] (\i-.5,\j-.5) -- (\i+.5,\j+.5);
    }
    \foreach \i/\j in {2/2,13/2} {
        \draw[black] (\i-.5,\j-.5) -- (\i+.5,\j+.5);
				\fill[white] (\i-.25,\j-.25) rectangle (\i+.25,\j+.25);
        \draw[black] (\i-.5,\j+.5) -- (\i+.5,\j-.5);
    }
    \foreach \i/\j in {4/2,5/1} {
        \draw[thick, red] (\i-.5,\j-.5) -- (\i+.5,\j+.5);
				\fill[white] (\i-.25,\j-.25) rectangle (\i+.25,\j+.25);
        \draw[thick, red] (\i-.5,\j+.5) -- (\i+.5,\j-.5);
    }
    \foreach \i/\j in {6/2,7/1,16/1,17/2} {
        \draw[black] (\i-.5,\j-.5) -- (\i+.5,\j+.5);
        \draw[black] (\i-.5,\j+.5) -- (\i+.5,\j-.5);
    }
        \draw[thick, red] (14.5,2.5) -- (15.5,2.5);
\end{tikzpicture}
\end{center}
\caption{Fix the left-hand side of the knot.  Grasp the right-hand side of the knot and rotate forward or backward (depending on the crossing information) by $180^\circ$ to replace two red crossings in the reduced billiard table diagram $\widetilde{D}$ obtained from a run of two in $w$ with a single crossing in $D$, as in \cite[Figure 3]{CoEZKr}.}
\label{fig:two-to-one}
\end{figure}
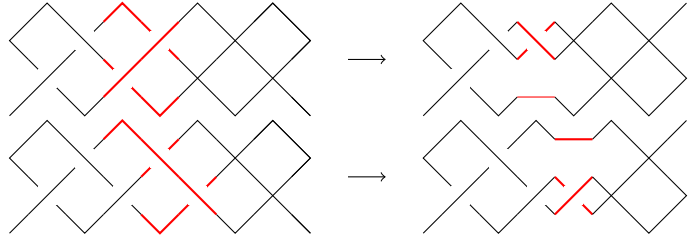

Figure \ref{fig:two-to-one} shows only two of four possible pictures:  with a double $++$ as $\sigma_1\sigma_2$ and with a double $-{}-$ as $\sigma_2^{-1}\sigma_1^{-1}$.  Omitted are the cases with a double $++$ as $\sigma_2\sigma_1$ and with a double $-{}-$ as $\sigma_1^{-1}\sigma_2^{-1}$.  These cases cannot occur as long as these moves are performed from left to right.  That is, the first run of two from the left will have its first letter in the correct place by sign:  $+$ yielding $\sigma_1$ and $-$ yielding $\sigma_2^{-1}$.

Finally we conclude with the facts that the first crossing is always positive here, so that the ``long'' strand to the left comes from an overcrossing, and that the last crossing is positive if the crossing number $c$ is odd and negative if $c$ is even, so that in both cases the ``long'' strand to the right comes from an undercrossing.

This guarantees that the crossings shown in this long knot are alternating with respect to each other.
\end{proof}

It is important to note that these alternating diagrams $S$ are no longer of the billiard table form.  Table \ref{tab:wtoD} explains this correspondence between $w$ and $D$.

\begin{center}
\begin{table}[h]
\begin{tabular}{|c||c|c|c|c|}
\hline
&&&&\\
Run in reduced billiard table word $w$ 				& $(+)^1$			& $(+)^2$					& $(-)^1$					& $(-)^2$			\\
&&&&\\
\hline
&&&&\\
Crossing in alternating diagram $D$		& $\sigma_1$	& $\sigma_2^{-1}$	& $\sigma_2^{-1}$	& $\sigma_1$	\\
&&&&\\
&& $\crossneg$ & $\crossneg$ &\\
&$\crosspos$ &&& $\crosspos$ \\
&&&&\\
\hline
\end{tabular}
\caption{Following Figure \ref{fig:two-to-one}, each run in the reduced billiard table word $w$ is converted to a single crossing in the alternating diagram $D$.}
\label{tab:wtoD}
\end{table}
\end{center}

To see many examples, the reader is encouraged to skip ahead to Table \ref{tab:c6} from Example \ref{ex:c6} and Table \ref{tab:c7} from Example \ref{ex:c7} to see how the billiard table word $w$ becomes an alternating word with an alternating diagram in the first three columns of each table.

To conclude this section, we count the number of knots appearing on our list of knots for a given crossing number $c$.

The following result is discussed in Guichard \cite{Gui} and more recently McCleary \cite[``Indicators'' p.206]{McC:binom}:
\begin{proposition}
\label{prop:Net}
(Netto \cite[p.20]{Net})
\begin{equation}
\binom{k}{0}+\binom{k}{3}+\binom{k}{6}+\ldots=\frac{1}{3}\left(2^k+2\cos \hspace{5mm}\frac{k\pi}{3}\hspace{5mm}\right)
\end{equation}
\begin{equation}
\binom{k}{1}+\binom{k}{4}+\binom{k}{7}+\ldots=\frac{1}{3}\left(2^k+2\cos \frac{(k-2)\pi}{3}\right)
\end{equation}
\begin{equation}
\binom{k}{2}+\binom{k}{5}+\binom{k}{8}+\ldots=\frac{1}{3}\left(2^k+2\cos \frac{(k-4)\pi}{3}\right)
\end{equation}
\end{proposition}

\begin{theorem}
\label{thm:number}
Consider the collection of \emph{all} 2-bridge knots of given crossing number $c$, counted exactly once or twice according to Assumption \ref{ass:model} and Remark \ref{rem:palindrome}, by considering every case of $\varepsilon_i\in\{1,2\}$ for $2\leq i \leq c-1$ and $\varepsilon_1=1=\varepsilon_c$ in Equations (\ref{eq:oddword}) and (\ref{eq:evenword}):
$$(+)^{\varepsilon_1}(-)^{\varepsilon_2}(+)^{\varepsilon_3}(-)^{\varepsilon_4}\ldots(-)^{\varepsilon_{c-1}}(+)^{\varepsilon_c} \text{ for $c$ odd and}$$
$$(+)^{\varepsilon_1}(-)^{\varepsilon_2}(+)^{\varepsilon_3}(-)^{\varepsilon_4}\ldots(+)^{\varepsilon_{c-1}}(-)^{\varepsilon_c} \text{ for $c$ even},$$
with reduced length $\ell=\sum_{i=1}^c \varepsilon_i \equiv 1 \mod 3$.

Then there are exactly $\frac{2^{c-2}+*}{3}$ (or approximately $\frac{2^{c-2}}{3}$) elements in this collection, where
\begin{equation*}
*=\begin{cases}
2\cos \hspace{5mm}\frac{(c-2)\pi}{3} & \text{if $c\equiv1$ mod $3$},\\
2\cos \frac{(c-4)\pi}{3} & \text{if $c\equiv0$ mod $3$, and}\\
2\cos \frac{(c-6)\pi}{3} & \text{if $c\equiv2$ mod $3$.}
\end{cases}
\end{equation*}
\end{theorem}

\begin{proof}
The total length of the word $\ell=\sum_{i=1}^c \varepsilon_i$ is equal to the crossing number $c$ plus the number of doubles (where $\varepsilon_i=2$).  Since the length must be congruent to 1 modulo 3: if $c$ is congruent to 0, then the number of doubles must be congruent to 1; if $c$ is congruent to 1, then the number of doubles must be congruent to 0; and if $c$ is congruent to 2, then the number of doubles must be congruent to 2.

Consider first the case where $c$ is congruent to 1.  One element from this collection is when 0 doubles are added.  We may also add 3 doubles, and these doubles must be selected from $2\leq i \leq c-1$, so there are $\binom{c-2}{3}$ elements.  We may also add 6 doubles for an additional $\binom{c-2}{6}$ elements, and so on.  Apply Proposition \ref{prop:Net}.  The other cases are similar.
\end{proof}

Note that this counting of doubles will reappear in the Main Theorem \ref{thm:main}.

\section{Counting the number of Seifert circles}
\label{sec:Seifert}

Seifert's algorithm first requires us to smooth the crossings of a diagram following their orientation.

\begin{definition}
\label{def:VH}
We say a crossing is \emph{vertically-smoothed} if its strands are both oriented upwards or both oriented downwards.  We denote such a crossing by $V$.

We say a crossing is \emph{horizontally-smoothed} if its strands are both oriented to the left or both oriented to the right.  We denote such a crossing by $H$.
\end{definition}

We adapt for this setting a more general result from the discussion of the Jones polynomials of 2-bridge knots.

\begin{proposition}
\label{prop:billiardwrithe}
(Cohen \cite[follows from Property 7.1]{Co:3bridge})  
In a billiard table diagram with $a=3$ and $n=3m+1$ crossings numbered from left to right, the orientations of the crossings are $H(VVH)^m$.

In a billiard table diagram with $a=3$ and $n=3m$ crossings numbered from left to right, the orientations of the crossings are $(VHV)^m$.  
\end{proposition}

We will only be concerned with the first case following Assumption \ref{ass:model}.  One can see this result immediately in Figure \ref{fig:HVVH}.

\begin{figure}[h]
\begin{center}
\begin{tikzpicture}[scale=.75]
    \foreach \i in {0,...,8} {
        \draw [very thin,dashed,gray] (\i,1) -- (\i,4);
    }
    \foreach \i in {1,...,4} {
        \draw [very thin,dashed,gray] (0,\i) -- (8,\i);
    }
\draw[thick](-.25,.75) -- (3,4) -- (6,1) -- (8,3) -- (7,4) -- (4,1) -- (1,4) -- (0,3) -- (2,1) -- (5,4) -- (8,1);
    \foreach \i / \j in {0/1,3/2,6/1} {
        \draw [thick, ->] (\i+.5,\j+.5) -- (\i+1.5,\j+1.5);
        \draw [thick, ->] (\i+.5,\j+1.5) -- (\i+1.5,\j+.5);
    }
    \foreach \i / \j in {1/2,2/1} {
        \draw [thick, ->] (\i+.5,\j+.5) -- (\i+1.5,\j+1.5);
        \draw [thick, <-] (\i+.5,\j+1.5) -- (\i+1.5,\j+.5);
    }
    \foreach \i / \j in {4/1,5/2} {
        \draw [thick, <-] (\i+.5,\j+.5) -- (\i+1.5,\j+1.5);
        \draw [thick, ->] (\i+.5,\j+1.5) -- (\i+1.5,\j+.5);
    }
    \foreach \i in {9,...,11} {
        \draw [very thin,dashed,lightred] (\i,1) -- (\i,4);
    }
    \foreach \i in {1,...,4} {
        \draw [very thin,dashed,lightred] (8,\i) -- (11,\i);
    }
\draw[thick, red] (8,1) -- (11.25, 4.25);
\draw[thick, red] (8,3) -- (9,4) -- (11,2) -- (10,1) -- (8,3);
    \foreach \i / \j in {9/2} {
        \draw [thick, red, ->] (\i+.5,\j+.5) -- (\i+1.5,\j+1.5);
        \draw [thick, red, ->] (\i+.5,\j+1.5) -- (\i+1.5,\j+.5);
    }
    \foreach \i / \j in {7/2} {
        \draw [thick, red, ->] (\i+1,\j+1) -- (\i+1.5,\j+1.5);
        \draw [thick, <-] (\i+.5,\j+1.5) -- (\i+1,\j+1);
		}
    \foreach \i / \j in {8/1} {
        \draw [thick, red, ->] (\i+.5,\j+.5) -- (\i+1.5,\j+1.5);
        \draw [thick, red, <-] (\i+.5,\j+1.5) -- (\i+1.5,\j+.5);
    }
\end{tikzpicture}
\caption{\label{fig:HVVH} The pattern for horizontal and vertical orientations for $n\equiv 1 \mod 3$ is not dependent on the billiard table word $w$.}
\end{center}
\end{figure}
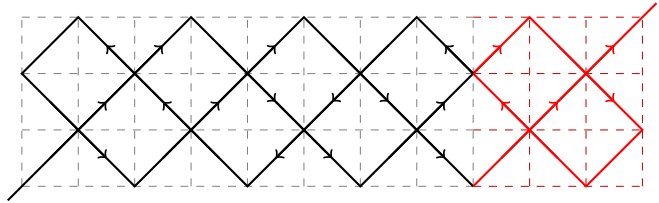

We next translate this result to our new setting of the alternating diagram obtained from the reduced billiard table word as in Theorem \ref{thm:alt}.

\begin{lemma}
\label{lem:contributions}
If either a single $+$ or a single $-$ appears in any position in the reduced billiard table word $w$ that is congruent to 1 modulo 3, then it corresponds to a single crossing in the alternating diagram $D$ that is horizontally-smoothed.

If either a double $++$ or a double $-{}-$ appears in any two positions in the reduced billiard table word $w$ that are congruent to 2 and 3 modulo 3, then they correspond to a single crossing in the alternating diagram $D$ that is horizontally-smoothed.

The remaining cases correspond to crossings in $D$ that are vertically-smoothed.
\end{lemma}

\begin{proof}
The first statement on a single $+$ or a single $-$ is easy to see based on Proposition \ref{prop:billiardwrithe}.

To show the second statement, we consider all of the possible cases for a double $++$ and a double $-{}-$.  First we remind the reader that a double $++$ must occur as $\sigma_1\sigma_2$ and a double $-{}-$ must occur as $\sigma_2^{-1}\sigma_1^{-1}$ as in the proof of Theorem \ref{thm:alt}, reducing the number of cases.  Note that both of these cases involve an ``overstrand'' passing over the two consecutive crossings.

Recall that our long knot starts on the left, moves to the right, backtracks to the left, and then returns to the right.  So if this overstrand is oriented to the left, then both of the other arcs must be oriented to the right.  This occurs in the third row of Figure \ref{fig:orientedDouble}.  If this overstrand is oriented to the right, then one of the other arcs must be oriented to the right and the other to the left.  These two cases occur in the first and second rows of Figure \ref{fig:orientedDouble}.

In the first row a $VH$ becomes a $V$ for both the $++$ and $-{}-$ cases.  In the second row an $HV$ becomes a $V$ for both cases.  In the third row a $VV$ becomes an $H$ for both cases.  These cases are summarized in Table \ref{tab:fig:orientedDouble}.  Lastly we note that $VV$ only occurs in positions congruent to 2 and 3 modulo 3 by Proposition \ref{prop:billiardwrithe}.
\end{proof}

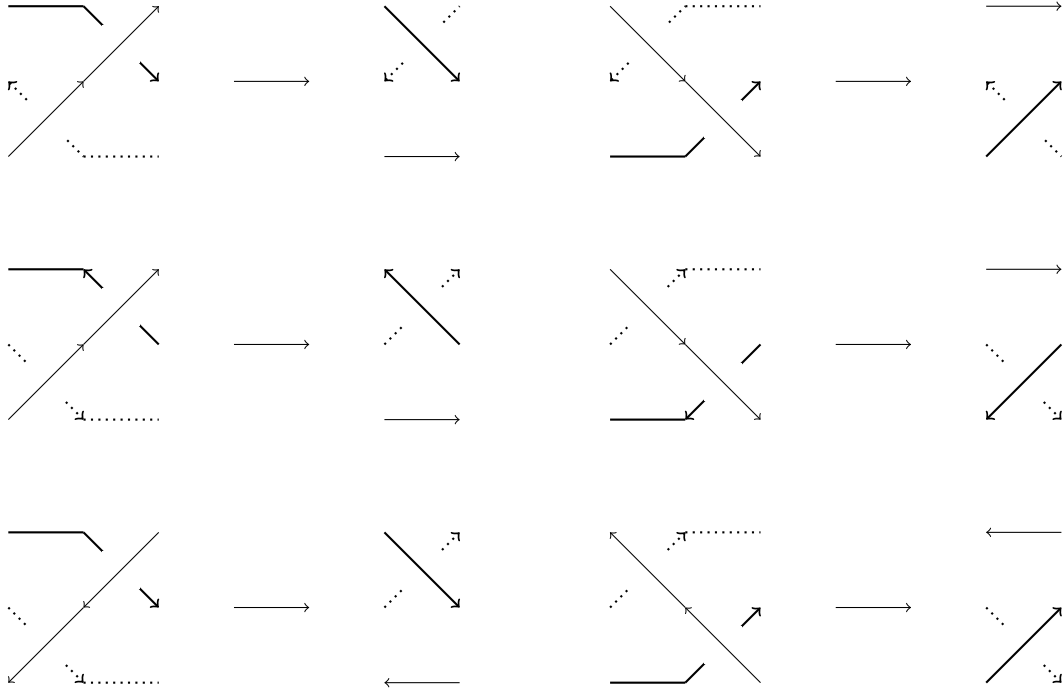
\begin{figure}[h]
\begin{center}
\begin{tikzpicture}[scale=1]

\draw[thick, dotted, ->] (0,1) -- (1,0);
\draw[thick, ->] (1,2) -- (2,1);
\fill[white] (.25,.25) rectangle (.75,.75);
\fill[white] (1.25,1.25) rectangle (1.75,1.75);
\draw[<-] (0,0) -- (1,1);
\draw[<-] (1,1) -- (2,2);
\draw[thick, dotted, -] (1,0) -- (2,0);
\draw[thick, -] (0,2) -- (1,2);

\draw[->] (3,1) -- (4,1);

\draw[thick, dotted, ->] (5,1) -- (6,2);
\fill[white] (5.25,1.25) rectangle (5.75,1.75);
\draw[thick, ->] (5,2) -- (6,1);
\draw[<-] (5,0) -- (6,0);

\begin{scope}[shift={(0,3.5)}]
\draw[thick, dotted, ->] (0,1) -- (1,0);
\draw[thick, <-] (1,2) -- (2,1);
\fill[white] (.25,.25) rectangle (.75,.75);
\fill[white] (1.25,1.25) rectangle (1.75,1.75);
\draw[->] (0,0) -- (1,1);
\draw[->] (1,1) -- (2,2);
\draw[thick, dotted, -] (1,0) -- (2,0);
\draw[thick, -] (0,2) -- (1,2);

\draw[->] (3,1) -- (4,1);

\draw[thick, dotted, ->] (5,1) -- (6,2);
\fill[white] (5.25,1.25) rectangle (5.75,1.75);
\draw[thick, <-] (5,2) -- (6,1);
\draw[->] (5,0) -- (6,0);
\end{scope}

\begin{scope}[shift={(0,7)}]
\draw[thick, dotted, <-] (0,1) -- (1,0);
\draw[thick, ->] (1,2) -- (2,1);
\fill[white] (.25,.25) rectangle (.75,.75);
\fill[white] (1.25,1.25) rectangle (1.75,1.75);
\draw[->] (0,0) -- (1,1);
\draw[->] (1,1) -- (2,2);
\draw[thick, dotted, -] (1,0) -- (2,0);
\draw[thick, -] (0,2) -- (1,2);

\draw[->] (3,1) -- (4,1);

\draw[thick, dotted, <-] (5,1) -- (6,2);
\fill[white] (5.25,1.25) rectangle (5.75,1.75);
\draw[thick, ->] (5,2) -- (6,1);
\draw[->] (5,0) -- (6,0);
\end{scope}

\begin{scope}[shift={(8,0)}]
\draw[thick, dotted, ->] (0,1) -- (1,2);
\draw[thick, ->] (1,0) -- (2,1);
\fill[white] (.25,1.25) rectangle (.75,1.75);
\fill[white] (1.25,.25) rectangle (1.75,.75);
\draw[<-] (0,2) -- (1,1);
\draw[<-] (1,1) -- (2,0);
\draw[thick, dotted, -] (1,2) -- (2,2);
\draw[thick, -] (0,0) -- (1,0);

\draw[->] (3,1) -- (4,1);

\draw[thick, dotted, ->] (5,1) -- (6,0);
\fill[white] (5.25,.25) rectangle (5.75,.75);
\draw[thick, ->] (5,0) -- (6,1);
\draw[<-] (5,2) -- (6,2);
\end{scope}

\begin{scope}[shift={(8,3.5)}]
\draw[thick, dotted, ->] (0,1) -- (1,2);
\draw[thick, <-] (1,0) -- (2,1);
\fill[white] (.25,1.25) rectangle (.75,1.75);
\fill[white] (1.25,.25) rectangle (1.75,.75);
\draw[->] (0,2) -- (1,1);
\draw[->] (1,1) -- (2,0);
\draw[thick, dotted, -] (1,2) -- (2,2);
\draw[thick, -] (0,0) -- (1,0);

\draw[->] (3,1) -- (4,1);

\draw[thick, dotted, ->] (5,1) -- (6,0);
\fill[white] (5.25,.25) rectangle (5.75,.75);
\draw[thick, <-] (5,0) -- (6,1);
\draw[->] (5,2) -- (6,2);
\end{scope}

\begin{scope}[shift={(8,7)}]
\draw[thick, dotted, <-] (0,1) -- (1,2);
\draw[thick, ->] (1,0) -- (2,1);
\fill[white] (.25,1.25) rectangle (.75,1.75);
\fill[white] (1.25,.25) rectangle (1.75,.75);
\draw[->] (0,2) -- (1,1);
\draw[->] (1,1) -- (2,0);
\draw[thick, dotted, -] (1,2) -- (2,2);
\draw[thick, -] (0,0) -- (1,0);

\draw[->] (3,1) -- (4,1);

\draw[thick, dotted, <-] (5,1) -- (6,0);
\fill[white] (5.25,.25) rectangle (5.75,.75);
\draw[thick, ->] (5,0) -- (6,1);
\draw[->] (5,2) -- (6,2);
\end{scope}

\end{tikzpicture}
\caption{\label{fig:orientedDouble} The only possible configurations of pairs of oriented crossings in a reduced billiard table diagram $\widetilde{D}$ obtained from doubles $++$ or $-{}-$ in the associated reduced billiard table word $w$, together with the corresponding single crossing in the alternating diagram $D$, as in the proof of Lemmma \ref{lem:contributions}.}
\end{center}
\end{figure}

\renewcommand{\arraystretch}{2}
\begin{table}[h]
\begin{center}
\begin{tabular}{|ccc|ccc|}
\hline
$++$ & $\rightarrow$ & $\sigma_2^{-1}$ & $-{}-$ & $\rightarrow$ & $\sigma_1$ \\
\hline
VH & $\rightarrow$ & V & VH & $\rightarrow$ & V \\
HV & $\rightarrow$ & V & HV & $\rightarrow$ & V \\
VV & $\rightarrow$ & H & VV & $\rightarrow$ & H \\
\hline
\end{tabular}
\caption{A summary of the cases associated with Figure \ref{fig:orientedDouble} in the proof of Lemma \ref{lem:contributions}.}
\label{tab:fig:orientedDouble}
\end{center}
\end{table}

We now use these vertically--smoothed crossings to count the number of Seifert circles obtained from the alternating diagram $D$.

\begin{definition}
\label{def:vert}
Consider the alternating diagram $D$ obtained from a reduced billiard table word $w$ as described in Section \ref{sec:model}.

We say a vertically-smoothed crossing is \emph{viable} if the next vertically-smoothed crossing appears at the same height or if it is the last vertically-smoothed crossing.

We say a vertically-smoothed crossing is \emph{sequential} if it is not the last vertically-smoothed crossing and if the immediate next crossing is also vertically-smoothed and appears at the same height.
\end{definition}

\begin{theorem}
\label{thm:Seifert}
The number of Seifert circles in the alternating diagram $D$ obtained from a reduced billard table word $w$ is two more than the number of viable vertically-smoothed crossings.
\end{theorem}

From this Lemma we can obtain an upper bound and a lower bound on the number of Seifert circles:

\begin{corollary}
\label{cor:upper}
The number of Seifert circles in the alternating diagram $D$ obtained from a reduced billard table word $w$ is \emph{at most} two more than the number of \emph{all} vertically-smoothed crossings.
\end{corollary}

\begin{corollary}
\label{cor:lower}
The number of Seifert circles in the alternating diagram $D$ obtained from a reduced billard table word $w$ is \emph{at least} two more than the number of sequential vertically-smoothed crossings.
\end{corollary}

\begin{remark}
\label{rem:goodbounds}
Observe that for every non-viable vertically-smoothed crossing, there is a viable one immediately following, and so the number of non-viable ones must be less than or equal to the number of viable ones.  Thus the total number of all vertically-smoothed crossings, which is the sum of the number of non-viable ones and the number of viable ones, must be less than or equal to twice the number of viable vertically-smoothed crossings.  This total number is also greater than or equal to the number of viable vertically-smoothed crossings.

This gives an idea of how good the upper bound of Corollary \ref{cor:upper} is.
\end{remark}

\medskip

Before we get to the proof of Theorem \ref{thm:Seifert}, we address the following small technical lemma that will make the main proof easier to follow.

\begin{lemma}
\label{lem:allH}
Suppose that one of our alternating diagrams $D$ has no vertically-smoothed crossings.  Then the number of crossings $c$ must be odd.

Furthermore, the first vertically-smoothed crossing, if it exists, must be a $\sigma_2^{-1}$.  If $c$ is odd, then the last vertically-smoothed crossing must be, as well; if $c$ is even, then the last must be a $\sigma_1$.
\end{lemma}

\begin{proof}
The first crossing is always $\sigma_1$ arising from a single $+$ appearing in position 1 (and hence 1 modulo 3) so that it is horizontally-smoothed.

If the second crossing was a $\sigma_2^{-1}$ arising from a single $-$ appearing in position 2 (and hence 2 modulo 3), it would be vertically-smoothed by Lemma \ref{lem:contributions}, so instead it must be a $\sigma_1$ arising from a double $-{}-$ appearing in positions 2 and 3, and it is horizontally-smoothed.

If the third crossing was a $\sigma_2^{-1}$ arising from a double $++$ appearing in positions 4 and 5, it would be vertically-smoothed by Lemma \ref{lem:contributions}, so instead it must be a $\sigma_1$ arising from a single $+$ appearing in position 4, and it is horizontally-smoothed.

For position 5, this brings us back to the argument for the second crossing, and so we must have $+$ and then $-{}-$ repeating.  By Assumption \ref{ass:model}, the reduced word must end with a single, so the word must be $+(-{}-+)^m$ for some positive integer $m$.  Then the alternating knot must be $\sigma_1^{2m+1}$, and so it must have an odd number of crossings.

These exact arguments also show the first vertically-smoothed crossing must be a $\sigma_2^{-1}$.

To see the final claims, consider the reverse of the reduced word.  If $c$ is odd, then the reverse also begins with $+$, and its alternating diagram must have a $\sigma_2^{-1}$ as its first vertically-smoothed crossing, so the original reduced word must have a $\sigma_2^{-1}$ as its last vertically-smoothed crossing.  If $c$ is even, then the reverse begins with a $-$ and so the mirror must be taken, and its alternating diagram must have a $\sigma_2^{-1}$ as its first vertically-smoothed crossing, so the original reduced word must have a $\sigma_1$ as its last vertically-smoothed crossing.
\end{proof}

\begin{proof}[Proof of Theorem \ref{thm:Seifert}]
First observe that a horizontal smoothing acts like the identity tangle in a braid.  Let us then ignore these crossings from our discussion below so that we have only vertically-smoothed crossings. 

Suppose there are no vertically-smoothed crossings.  Then by Lemma \ref{lem:allH}, the crossing number of the alternating diagram $D$ must be odd, and so this gives two circles: one from the plat closure between the strands labeled two and three, and one from the ``long'' component of the long knot.

Suppose now there \emph{are} vertically-smoothed crossings, and let's consider the base cases determined by whether $c$ is odd or even.  By Lemma \ref{lem:allH}, if $c$ is odd, the first and the last vertically-smoothed crossings are a $\sigma_2^{-1}$.  If these are the same crossing, this gives the case shown in the first row of Figure \ref{fig:basecases}, where a smoothing yields three circles, satisfying the conclusion of the Theorem.  If these are \emph{not} the same crossing, this gives the case shown in the second row of Figure \ref{fig:basecases}, where a smoothing yields four circles, satisfying the conclusion of the Theorem.  By Lemma \ref{lem:allH}, if $c$ is even, the first vertically-smoothed crossing is a $\sigma_2^{-1}$ and the last is a $\sigma_1$.  That gives the case shown in the third row of Figure \ref{fig:basecases}, where a smoothing yields three circles, satisfying the conclusion of the Theorem since the first crossing is not viable.

\begin{figure}[h!]
\begin{center}
\begin{tikzpicture}[scale=.5]

\begin{scope}[shift={(0,7)}]
\draw[-] (-1,0) -- (6,0);
\draw[-] (0,1) -- (5,1);
\draw[-] (0,2) -- (5,2);
\draw[-] (0,2) arc (90:270:.5);
\draw[-] (5,1) arc (-90:90:.5);
\foreach \x/ \y in {2/1}
    {
    \draw[-, dotted, thick] (\x,\y) .. controls (\x+.33,\y+.33) and (\x+.33,\y+.66) .. (\x,\y+1);
    \draw[-, dotted, thick] (\x+1,\y) .. controls (\x+.66,\y+.33) and (\x+.66,\y+.66) .. (\x+1,\y+1);
    }
\draw[->] (7,1) -- (8,1);
\end{scope}

\begin{scope}[shift={(10,7)}]
\draw[-] (-1,0) -- (6,0);
\draw[-] (0,1) -- (2,1);
\draw[-] (3,1) -- (5,1);
\draw[-] (0,2) -- (2,2);
\draw[-] (3,2) -- (5,2);
\draw[-] (0,2) arc (90:270:.5);
\draw[-] (5,1) arc (-90:90:.5);
\foreach \x/ \y in {2/1}
    {
    \draw[-] (\x,\y) .. controls (\x+.33,\y+.33) and (\x+.33,\y+.66) .. (\x,\y+1);
    \draw[-] (\x+1,\y) .. controls (\x+.66,\y+.33) and (\x+.66,\y+.66) .. (\x+1,\y+1);
    }
\foreach \x/ \y in {2/1}
	{
\draw[black, ultra thick] (\x+1,\y) .. controls (\x+.66,\y+.33) and (\x+.66,\y+.66) .. (\x+1,\y+1);
	}
\draw[black, ultra thick] (-1,0) -- (0,0);
\draw[black, ultra thick] (0,2) arc (90:270:.5);
\end{scope}

\begin{scope}[shift={(0,3.5)}]
\draw[-] (-1,0) -- (6,0);
\draw[-] (0,1) -- (5,1);
\draw[-] (0,2) -- (5,2);
\draw[-] (0,2) arc (90:270:.5);
\draw[-] (5,1) arc (-90:90:.5);
\foreach \x/ \y in {1/1, 3/1}
    {
    \draw[-, dotted, thick] (\x,\y) .. controls (\x+.33,\y+.33) and (\x+.33,\y+.66) .. (\x,\y+1);
    \draw[-, dotted, thick] (\x+1,\y) .. controls (\x+.66,\y+.33) and (\x+.66,\y+.66) .. (\x+1,\y+1);
    }
\draw[->] (7,1) -- (8,1);
\end{scope}

\begin{scope}[shift={(10,3.5)}]
\draw[-] (-1,0) -- (6,0);
\draw[-] (0,1) -- (1,1);
\draw[-] (2,1) -- (3,1);
\draw[-] (4,1) -- (5,1);
\draw[-] (0,2) -- (1,2);
\draw[-] (2,2) -- (3,2);
\draw[-] (4,2) -- (5,2);
\draw[-] (0,2) arc (90:270:.5);
\draw[-] (5,1) arc (-90:90:.5);
\foreach \x/ \y in {1/1, 3/1}
    {
    \draw[-] (\x,\y) .. controls (\x+.33,\y+.33) and (\x+.33,\y+.66) .. (\x,\y+1);
    \draw[-] (\x+1,\y) .. controls (\x+.66,\y+.33) and (\x+.66,\y+.66) .. (\x+1,\y+1);
    }
\foreach \x/ \y in {1/1, 3/1}
	{
\draw[black, ultra thick] (\x+1,\y) .. controls (\x+.66,\y+.33) and (\x+.66,\y+.66) .. (\x+1,\y+1);
	}
\draw[black, ultra thick] (-1,0) -- (0,0);
\draw[black, ultra thick] (0,2) arc (90:270:.5);
\end{scope}

\begin{scope}[shift={(0,0)}]
\draw[-] (-1,0) -- (5,0);
\draw[-] (0,1) -- (5,1);
\draw[-] (0,2) -- (6,2);
\draw[-] (0,2) arc (90:270:.5);
\draw[-] (5,0) arc (-90:90:.5);
\foreach \x/ \y in {1/1, 3/0}
    {
    \draw[-, dotted, thick] (\x,\y) .. controls (\x+.33,\y+.33) and (\x+.33,\y+.66) .. (\x,\y+1);
    \draw[-, dotted, thick] (\x+1,\y) .. controls (\x+.66,\y+.33) and (\x+.66,\y+.66) .. (\x+1,\y+1);
    }

\draw[->] (7,1) -- (8,1);
\end{scope}

\begin{scope}[shift={(10,0)}]
\draw[-] (-1,0) -- (3,0);
\draw[-] (4,0) -- (5,0);
\draw[-] (0,1) -- (1,1);
\draw[-] (2,1) -- (3,1);
\draw[-] (4,1) -- (5,1);
\draw[-] (0,2) -- (1,2);
\draw[-] (2,2) -- (6,2);
\draw[-] (0,2) arc (90:270:.5);
\draw[-] (5,0) arc (-90:90:.5);
\foreach \x/ \y in {1/1, 3/0}
    {
    \draw[-] (\x,\y) .. controls (\x+.33,\y+.33) and (\x+.33,\y+.66) .. (\x,\y+1);
    \draw[-] (\x+1,\y) .. controls (\x+.66,\y+.33) and (\x+.66,\y+.66) .. (\x+1,\y+1);
    }
\foreach \x/ \y in {3/0}
	{
\draw[black, ultra thick] (\x+1,\y) .. controls (\x+.66,\y+.33) and (\x+.66,\y+.66) .. (\x+1,\y+1);
	}
\foreach \x/ \y in {1/1}
	{
\draw[lightgray, ultra thick] (\x+1,\y) .. controls (\x+.66,\y+.33) and (\x+.66,\y+.66) .. (\x+1,\y+1);
	}
\draw[black, ultra thick] (-1,0) -- (0,0);
\draw[black, ultra thick] (0,2) arc (90:270:.5);
\end{scope}

\end{tikzpicture}
\end{center}
\caption{Base cases for the proof of Theorem \ref{thm:Seifert}, where the bold indicates the start of Seifert circles.}
\label{fig:basecases}
\end{figure}
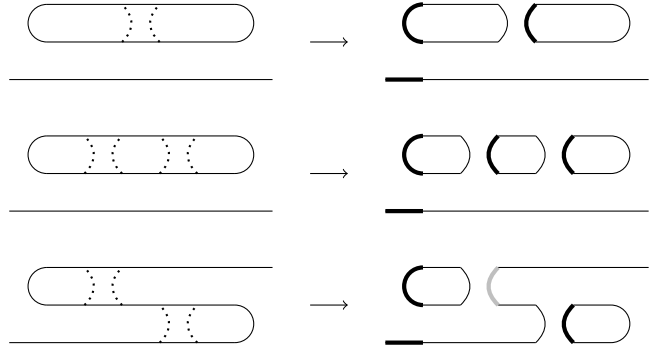

Next we consider the two cases of adding additional vertically-smoothed crossings:  whether the next crossing is at the same height as the previous one or whether it differs.

Suppose we have a vertically-smoothed crossing at a given height and that the next vertically-smoothed crossing is at the same height.  Then as in the first row of Figure \ref{fig:indstep}, we do indeed have an additional circle.

\begin{figure}[h!]
\begin{center}
\begin{tikzpicture}[scale=.5]

\begin{scope}[shift={(0,2.5)}]
\draw[-] (0,1) -- (1,1);
\draw[-] (2,1) -- (3,1);
\draw[-] (0,2) -- (1,2);
\draw[-] (2,2) -- (3,2);
\foreach \x/ \y in {1/1}
    {
    \draw[-] (\x,\y) .. controls (\x+.33,\y+.33) and (\x+.33,\y+.66) .. (\x,\y+1);
    \draw[-] (\x+1,\y) .. controls (\x+.66,\y+.33) and (\x+.66,\y+.66) .. (\x+1,\y+1);
    }
\foreach \x/ \y in {1/1}
	{
\draw[black, ultra thick] (\x+1,\y) .. controls (\x+.66,\y+.33) and (\x+.66,\y+.66) .. (\x+1,\y+1);
	}
\draw[->] (4,1.5) -- (5,1.5);
\end{scope}

\begin{scope}[shift={(6,2.5)}]
\draw[-] (0,1) -- (1,1);
\draw[-] (2,1) -- (3,1);
\draw[-] (4,1) -- (5,1);
\draw[-] (0,2) -- (1,2);
\draw[-] (2,2) -- (3,2);
\draw[-] (4,2) -- (5,2);
\foreach \x/ \y in {1/1, 3/1}
    {
    \draw[-] (\x,\y) .. controls (\x+.33,\y+.33) and (\x+.33,\y+.66) .. (\x,\y+1);
    \draw[-] (\x+1,\y) .. controls (\x+.66,\y+.33) and (\x+.66,\y+.66) .. (\x+1,\y+1);
    }
\foreach \x/ \y in {1/1, 3/1}
	{
\draw[black, ultra thick] (\x+1,\y) .. controls (\x+.66,\y+.33) and (\x+.66,\y+.66) .. (\x+1,\y+1);
	}
\end{scope}

\begin{scope}[shift={(0,0)}]
\draw[-] (0,0) -- (3,0);
\draw[-] (0,1) -- (1,1);
\draw[-] (2,1) -- (3,1);
\draw[-] (0,2) -- (1,2);
\draw[-] (2,2) -- (3,2);
\foreach \x/ \y in {1/1}
    {
    \draw[-] (\x,\y) .. controls (\x+.33,\y+.33) and (\x+.33,\y+.66) .. (\x,\y+1);
    \draw[-] (\x+1,\y) .. controls (\x+.66,\y+.33) and (\x+.66,\y+.66) .. (\x+1,\y+1);
    }
\foreach \x/ \y in {1/1}
	{
\draw[black, ultra thick] (\x+1,\y) .. controls (\x+.66,\y+.33) and (\x+.66,\y+.66) .. (\x+1,\y+1);
	}
\draw[->] (4,1) -- (5,1);
\end{scope}

\begin{scope}[shift={(6,0)}]
\draw[-] (0,0) -- (3,0);
\draw[-] (4,0) -- (5,0);
\draw[-] (0,1) -- (1,1);
\draw[-] (2,1) -- (3,1);
\draw[-] (4,1) -- (5,1);
\draw[-] (0,2) -- (1,2);
\draw[-] (2,2) -- (5,2);
\foreach \x/ \y in {1/1, 3/0}
    {
    \draw[-] (\x,\y) .. controls (\x+.33,\y+.33) and (\x+.33,\y+.66) .. (\x,\y+1);
    \draw[-] (\x+1,\y) .. controls (\x+.66,\y+.33) and (\x+.66,\y+.66) .. (\x+1,\y+1);
    }
\foreach \x/ \y in {3/0}
	{
\draw[black, ultra thick] (\x+1,\y) .. controls (\x+.66,\y+.33) and (\x+.66,\y+.66) .. (\x+1,\y+1);
	}
\foreach \x/ \y in {1/1}
	{
\draw[lightgray, ultra thick] (\x+1,\y) .. controls (\x+.66,\y+.33) and (\x+.66,\y+.66) .. (\x+1,\y+1);
	}
\end{scope}
\end{tikzpicture}
\end{center}
\caption{Inductive steps for the proof of Theorem \ref{thm:Seifert}, where the bold black indicates the start of Seifert circles and the bold gray indicates a vertically-smoothed crossing that is not viable.}
\label{fig:indstep}
\end{figure}
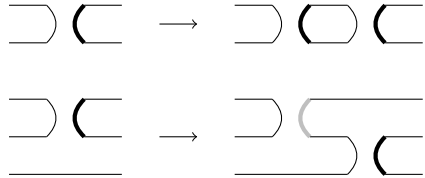

Suppose we have a vertically-smoothed crossing at a given height and that the next vertically-smoothed crossing is at the opposite height.  Then as in the second row of Figure \ref{fig:indstep}, we do \emph{not} have an additional circle, but as the first crossing is no longer viable, we have not added any additional viable vertically-smoothed crossings.
\end{proof}


\begin{example}
\label{ex:c6}
Let us consider the set of all knots obtained in this model for crossing number $c=6$.  Since $6\equiv 0$ modulo 3, we only consider adding $d\equiv 1$ modulo 3 doubles for $0\leq d \leq 6-2 = 4$, yielding $d=1$ or $d=4$.  The $\binom{4}{1}=4$ cases for $d=1$ and then the $\binom{4}{4}=1$ case for $d=4$ are listed in the first column of Table \ref{tab:c6}.  The knots are named in the final column.

In the second column, the reduced words are transformed into sigma notation (with a plat closure) for an alternating diagram following Theorem \ref{thm:alt}, and these alternating diagrams are presented in the third column.  The fourth column gives the oriented smoothing following Seifert's algorithm.

Marked in black in these columns are the letters or crossings or smoothings corresponding to viable vertical crossings.  Marked in gray are those corresponding to vertical crossings that are not viable.  One can see how each viable vertical crossing contributes a new circle and how each vertical non-viable crossing does not.

The actual average number $s$ of Seifert circles for $c=6$ is $2+\frac{9}{5}=\frac{19}{5}=3\frac{4}{5}$.

The upper bound on the average number of Seifert circles given by Theorem \ref{thm:Seifert} is $2+\frac{14}{5}=\frac{24}{5}=4\frac{4}{5}$.

By the genus formula for alternating knots, the actual average genus is $g(K)=1-\frac{1+s-c}{2}=\frac{1-s+c}{2}=\frac{1+c}{2}-\frac{s}{2}=\frac{7}{2}-\frac{19}{10}=\frac{16}{10}=1\frac{6}{10}$.

The lower bound on the average genus is $g(K)=\frac{7}{2}-\frac{24}{10}=\frac{11}{10}=1\frac{1}{10}$.
\end{example}


\begin{table}
\begin{tabular}{|l|c|c|c|c|}
\hline
Billiard Table &	Alternating Word 																														& Diagram $D$ & Seifert circles & Knot \\
\hspace{5mm} Word	$w$	 & & & & \hspace{1mm} $K$ \\
\hline
+-{}-+${\lightgray\Minus}\Plus$- 				& $\sigma_1^3\sigma_2^{-1}\sigma_1\sigma_2^{-1}$		& 
\begin{tikzpicture}[scale=.5]
\draw[->] (-1,0) -- (0,0);
\draw[->] (6,2) -- (7,2);
\draw[<-] (0,2) -- (3,2);
\draw[->] (3,0) -- (4,0);
\draw[->] (4,2) -- (5,2);
\draw[<-] (5,0) -- (6,0);
\draw[->] (0,2) arc (90:270:.5);
\draw[<-] (6,0) arc (-90:90:.5);
\foreach \x/ \y in {4/0}
    {
    \draw[->] (\x+1,\y) -- (\x,\y+1);
    \draw[color=white, line width=10] (\x,\y) -- (\x+1,\y+1);
    \draw[->] (\x,\y) -- (\x+1,\y+1);
    }
\foreach \x/ \y in {0/0, 1/0, 2/0}
    {
    \draw[<-] (\x+1,\y) -- (\x,\y+1);
    \draw[color=white, line width=10] (\x,\y) -- (\x+1,\y+1);
    \draw[->] (\x,\y) -- (\x+1,\y+1);
    }
\foreach \x/ \y in {3/1}
    {
    \draw[->] (\x,\y) -- (\x+1,\y+1);
    \draw[color=white, line width=10] (\x+1,\y) -- (\x,\y+1);
    \draw[->] (\x+1,\y) -- (\x,\y+1);
    }
\foreach \x/ \y in {5/1}
    {
    \draw[->] (\x,\y) -- (\x+1,\y+1);
    \draw[color=white, line width=10] (\x+1,\y) -- (\x,\y+1);
    \draw[<-] (\x+1,\y) -- (\x,\y+1);
    }
\foreach \x/ \y in {4/0}
	{
\filldraw[shift={(\x+.5,\y+.5)}] [fill=black, draw=black, rounded corners] (0,0) circle (5pt);
	}
\foreach \x/ \y in {3/1}
	{
\filldraw[shift={(\x+.5,\y+.5)}] [fill=lightgray, draw=lightgray, rounded corners] (0,0) circle (5pt);
	}
\end{tikzpicture}			&
\begin{tikzpicture}[scale=.5]
\draw[->] (-1,0) -- (0,0);
\draw[->] (6,2) -- (7,2);
\draw[->] (0,2) arc (90:270:.5);
\draw[<-] (6,0) arc (-90:90:.5);
\draw[<-] (0,2) -- (3,2);
\draw[->] (3,0) -- (4,0);
\draw[->] (4,2) -- (5,2);
\draw[<-] (5,0) -- (6,0);
\foreach \x/ \y in {3/1, 4/0}
    {
    \draw[->] (\x,\y) .. controls (\x+.33,\y+.33) and (\x+.33,\y+.66) .. (\x,\y+1);
    \draw[->] (\x+1,\y) .. controls (\x+.66,\y+.33) and (\x+.66,\y+.66) .. (\x+1,\y+1);
    }
\foreach \x/ \y in {0/0, 1/0, 2/0, 5/1}
    {
    \draw[->] (\x,\y) .. controls (\x+.33,\y+.33) and (\x+.66,\y+.33) .. (\x+1,\y);
    \draw[->] (\x,\y+1) .. controls (\x+.33,\y+.66) and (\x+.66,\y+.66) .. (\x+1,\y+1);
    }
\foreach \x/ \y in {4/0}
	{
\draw[black, ultra thick] (\x+1,\y) .. controls (\x+.66,\y+.33) and (\x+.66,\y+.66) .. (\x+1,\y+1);
	}
\foreach \x/ \y in {3/1}
	{
\draw[lightgray, ultra thick] (\x+1,\y) .. controls (\x+.66,\y+.33) and (\x+.66,\y+.66) .. (\x+1,\y+1);
	}
\draw[black, ultra thick] (-1,0) -- (0,0);
\draw[black, ultra thick] (0,2) arc (90:270:.5);
\end{tikzpicture}			&	$6_2$ \\
\hline
+$\Minus\Plus\Plus{\lightgray\Minus}\Plus$-	& $\sigma_1\sigma_2^{-3}\sigma_1\sigma_2^{-1}$					& 
\begin{tikzpicture}[scale=.5]
\draw[->] (-1,0) -- (0,0);
\draw[->] (6,2) -- (7,2);
\draw[->] (0,2) arc (90:270:.5);
\draw[<-] (6,0) arc (-90:90:.5);
\draw[<-] (0,2) -- (1,2);
\draw[->] (1,0) -- (4,0);
\draw[->] (4,2) -- (5,2);
\draw[<-] (5,0) -- (6,0);
\foreach \x/ \y in {4/0}
    {
    \draw[->] (\x+1,\y) -- (\x,\y+1);
    \draw[color=white, line width=10] (\x,\y) -- (\x+1,\y+1);
    \draw[->] (\x,\y) -- (\x+1,\y+1);
    }
\foreach \x/ \y in {0/0}
    {
    \draw[<-] (\x+1,\y) -- (\x,\y+1);
    \draw[color=white, line width=10] (\x,\y) -- (\x+1,\y+1);
    \draw[->] (\x,\y) -- (\x+1,\y+1);
    }
\foreach \x/ \y in {2/1}
    {
    \draw[<-] (\x,\y) -- (\x+1,\y+1);
    \draw[color=white, line width=10] (\x+1,\y) -- (\x,\y+1);
    \draw[<-] (\x+1,\y) -- (\x,\y+1);
    }
\foreach \x/ \y in {1/1, 3/1}
    {
    \draw[->] (\x,\y) -- (\x+1,\y+1);
    \draw[color=white, line width=10] (\x+1,\y) -- (\x,\y+1);
    \draw[->] (\x+1,\y) -- (\x,\y+1);
    }
\foreach \x/ \y in {5/1}
    {
    \draw[->] (\x,\y) -- (\x+1,\y+1);
    \draw[color=white, line width=10] (\x+1,\y) -- (\x,\y+1);
    \draw[<-] (\x+1,\y) -- (\x,\y+1);
    }
\foreach \x/ \y in {1/1, 2/1, 4/0}
	{
\filldraw[shift={(\x+.5,\y+.5)}] [fill=black, draw=black, rounded corners] (0,0) circle (5pt);
	}
\foreach \x/ \y in {3/1}
	{
\filldraw[shift={(\x+.5,\y+.5)}] [fill=lightgray, draw=lightgray, rounded corners] (0,0) circle (5pt);
	}
\end{tikzpicture}		&
\begin{tikzpicture}[scale=.5]
\draw[->] (-1,0) -- (0,0);
\draw[->] (6,2) -- (7,2);
\draw[->] (0,2) arc (90:270:.5);
\draw[<-] (6,0) arc (-90:90:.5);
\draw[<-] (0,2) -- (1,2);
\draw[->] (1,0) -- (4,0);
\draw[->] (4,2) -- (5,2);
\draw[<-] (5,0) -- (6,0);
\foreach \x/ \y in {1/1, 3/1, 4/0}
    {
    \draw[->] (\x,\y) .. controls (\x+.33,\y+.33) and (\x+.33,\y+.66) .. (\x,\y+1);
    \draw[->] (\x+1,\y) .. controls (\x+.66,\y+.33) and (\x+.66,\y+.66) .. (\x+1,\y+1);
    }
\foreach \x/ \y in {0/0, 5/1}
    {
    \draw[->] (\x,\y) .. controls (\x+.33,\y+.33) and (\x+.66,\y+.33) .. (\x+1,\y);
    \draw[->] (\x,\y+1) .. controls (\x+.33,\y+.66) and (\x+.66,\y+.66) .. (\x+1,\y+1);
    }
\foreach \x/ \y in {2/1}
    {
    \draw[<-] (\x,\y) .. controls (\x+.33,\y+.33) and (\x+.33,\y+.66) .. (\x,\y+1);
    \draw[<-] (\x+1,\y) .. controls (\x+.66,\y+.33) and (\x+.66,\y+.66) .. (\x+1,\y+1);
    }
\foreach \x/ \y in {1/1, 2/1, 4/0}
	{
\draw[black, ultra thick] (\x+1,\y) .. controls (\x+.66,\y+.33) and (\x+.66,\y+.66) .. (\x+1,\y+1);
	}
\foreach \x/ \y in {3/1}
	{
\draw[lightgray, ultra thick] (\x+1,\y) .. controls (\x+.66,\y+.33) and (\x+.66,\y+.66) .. (\x+1,\y+1);
	}
\draw[black, ultra thick] (-1,0) -- (0,0);
\draw[black, ultra thick] (0,2) arc (90:270:.5);
\end{tikzpicture}		& $6_1$ \\

\hline
+${\lightgray\Minus}\Plus\Minus\Minus\Plus$-	&	$\sigma_1\sigma_2^{-1}\sigma_1^3\sigma_2^{-1}$							& 
\begin{tikzpicture}[scale=.5]
\draw[->] (-1,0) -- (0,0);
\draw[->] (6,2) -- (7,2);
\draw[->] (0,2) arc (90:270:.5);
\draw[<-] (6,0) arc (-90:90:.5);
\draw[<-] (0,2) -- (1,2);
\draw[->] (1,0) -- (2,0);
\draw[->] (2,2) -- (5,2);
\draw[<-] (5,0) -- (6,0);
\foreach \x/ \y in {0/0, 5/1}
    {
    \draw[<-] (\x+1,\y) -- (\x,\y+1);
    \draw[color=white, line width=10] (\x,\y) -- (\x+1,\y+1);
    \draw[->] (\x,\y) -- (\x+1,\y+1);
    }
\foreach \x/ \y in {1/1, 2/0, 4/0}
    {
    \draw[->] (\x,\y) -- (\x+1,\y+1);
    \draw[color=white, line width=10] (\x+1,\y) -- (\x,\y+1);
    \draw[->] (\x+1,\y) -- (\x,\y+1);
    }
\foreach \x/ \y in {3/0}
    {
    \draw[<-] (\x,\y) -- (\x+1,\y+1);
    \draw[color=white, line width=10] (\x+1,\y) -- (\x,\y+1);
    \draw[<-] (\x+1,\y) -- (\x,\y+1);
    }
\foreach \x/ \y in {2/0, 3/0, 4/0}
	{
\filldraw[shift={(\x+.5,\y+.5)}] [fill=black, draw=black, rounded corners] (0,0) circle (5pt);
	}
\foreach \x/ \y in {1/1}
	{
\filldraw[shift={(\x+.5,\y+.5)}] [fill=lightgray, draw=lightgray, rounded corners] (0,0) circle (5pt);
	}
\end{tikzpicture}		&
\begin{tikzpicture}[scale=.5]
\draw[->] (-1,0) -- (0,0);
\draw[->] (6,2) -- (7,2);
\draw[->] (0,2) arc (90:270:.5);
\draw[<-] (6,0) arc (-90:90:.5);
\draw[<-] (0,2) -- (1,2);
\draw[->] (1,0) -- (2,0);
\draw[->] (2,2) -- (5,2);
\draw[<-] (5,0) -- (6,0);
\foreach \x/ \y in {0/0, 5/1}
    {
    \draw[->] (\x,\y) .. controls (\x+.33,\y+.33) and (\x+.66,\y+.33) .. (\x+1,\y);
    \draw[->] (\x,\y+1) .. controls (\x+.33,\y+.66) and (\x+.66,\y+.66) .. (\x+1,\y+1);
    }
\foreach \x/ \y in {1/1, 2/0, 4/0}
    {
    \draw[->] (\x,\y) .. controls (\x+.33,\y+.33) and (\x+.33,\y+.66) .. (\x,\y+1);
    \draw[->] (\x+1,\y) .. controls (\x+.66,\y+.33) and (\x+.66,\y+.66) .. (\x+1,\y+1);
    }
\foreach \x/ \y in {3/0}
    {
    \draw[<-] (\x,\y) .. controls (\x+.33,\y+.33) and (\x+.33,\y+.66) .. (\x,\y+1);
    \draw[<-] (\x+1,\y) .. controls (\x+.66,\y+.33) and (\x+.66,\y+.66) .. (\x+1,\y+1);
    }
\foreach \x/ \y in {2/0, 3/0, 4/0}
	{
\draw[black, ultra thick] (\x+1,\y) .. controls (\x+.66,\y+.33) and (\x+.66,\y+.66) .. (\x+1,\y+1);
	}
\foreach \x/ \y in {1/1}
	{
\draw[lightgray, ultra thick] (\x+1,\y) .. controls (\x+.66,\y+.33) and (\x+.66,\y+.66) .. (\x+1,\y+1);
	}
\draw[black, ultra thick] (-1,0) -- (0,0);
\draw[black, ultra thick] (0,2) arc (90:270:.5);
\end{tikzpicture}		& $6_1$ \\
\hline
+${\lightgray\Minus}\Plus$-++-	& $\sigma_1\sigma_2^{-1}\sigma_1\sigma_2^{-3}$												& 
\begin{tikzpicture}[scale=.5]
\draw[->] (-1,0) -- (0,0);
\draw[->] (6,2) -- (7,2);
\draw[->] (0,2) arc (90:270:.5);
\draw[<-] (6,0) arc (-90:90:.5);
\draw[<-] (0,2) -- (1,2);
\draw[->] (1,0) -- (2,0);
\draw[->] (2,2) -- (3,2);
\draw[<-] (3,0) -- (6,0);
\foreach \x/ \y in {0/0}
    {
    \draw[<-] (\x+1,\y) -- (\x,\y+1);
    \draw[color=white, line width=10] (\x,\y) -- (\x+1,\y+1);
    \draw[->] (\x,\y) -- (\x+1,\y+1);
    }
\foreach \x/ \y in {2/0}
    {
    \draw[->] (\x+1,\y) -- (\x,\y+1);
    \draw[color=white, line width=10] (\x,\y) -- (\x+1,\y+1);
    \draw[->] (\x,\y) -- (\x+1,\y+1);
    }
\foreach \x/ \y in {1/1}
    {
    \draw[->] (\x,\y) -- (\x+1,\y+1);
    \draw[color=white, line width=10] (\x+1,\y) -- (\x,\y+1);
    \draw[->] (\x+1,\y) -- (\x,\y+1);
    }
\foreach \x/ \y in {3/1, 4/1, 5/1}
    {
    \draw[->] (\x,\y) -- (\x+1,\y+1);
    \draw[color=white, line width=10] (\x+1,\y) -- (\x,\y+1);
    \draw[<-] (\x+1,\y) -- (\x,\y+1);
    }
\foreach \x/ \y in {2/0}
	{
\filldraw[shift={(\x+.5,\y+.5)}] [fill=black, draw=black, rounded corners] (0,0) circle (5pt);
	}
\foreach \x/ \y in {1/1}
	{
\filldraw[shift={(\x+.5,\y+.5)}] [fill=lightgray, draw=lightgray, rounded corners] (0,0) circle (5pt);
	}
\end{tikzpicture}		&
\begin{tikzpicture}[scale=.5]
\draw[->] (-1,0) -- (0,0);
\draw[->] (6,2) -- (7,2);
\draw[->] (0,2) arc (90:270:.5);
\draw[<-] (6,0) arc (-90:90:.5);
\draw[<-] (0,2) -- (1,2);
\draw[->] (1,0) -- (2,0);
\draw[->] (2,2) -- (3,2);
\draw[<-] (3,0) -- (6,0);
\foreach \x/ \y in {0/0}
    {
    \draw[->] (\x,\y) .. controls (\x+.33,\y+.33) and (\x+.66,\y+.33) .. (\x+1,\y);
    \draw[->] (\x,\y+1) .. controls (\x+.33,\y+.66) and (\x+.66,\y+.66) .. (\x+1,\y+1);
    }
\foreach \x/ \y in {1/1, 2/0}
    {
    \draw[->] (\x,\y) .. controls (\x+.33,\y+.33) and (\x+.33,\y+.66) .. (\x,\y+1);
    \draw[->] (\x+1,\y) .. controls (\x+.66,\y+.33) and (\x+.66,\y+.66) .. (\x+1,\y+1);
    }
\foreach \x/ \y in {3/1, 4/1, 5/1}
    {
    \draw[->] (\x,\y) .. controls (\x+.33,\y+.33) and (\x+.66,\y+.33) .. (\x+1,\y);
    \draw[->] (\x,\y+1) .. controls (\x+.33,\y+.66) and (\x+.66,\y+.66) .. (\x+1,\y+1);
    }
\foreach \x/ \y in {1/1}
	{
\draw[lightgray, ultra thick] (\x+1,\y) .. controls (\x+.66,\y+.33) and (\x+.66,\y+.66) .. (\x+1,\y+1);
	}
\foreach \x/ \y in {2/0}
	{
\draw[black, ultra thick] (\x+1,\y) .. controls (\x+.66,\y+.33) and (\x+.66,\y+.66) .. (\x+1,\y+1);
	}
\draw[black, ultra thick] (-1,0) -- (0,0);
\draw[black, ultra thick] (0,2) arc (90:270:.5);
\end{tikzpicture}		& $6_2$ \\
\hline
+-{}-${\lightgray\Plus\Plus}\Minus\Minus$++-	&	$\sigma_1^2\sigma_2^{-1}\sigma_1\sigma_2^{-2}$						& 
\begin{tikzpicture}[scale=.5]
\draw[->] (-1,0) -- (0,0);
\draw[->] (6,2) -- (7,2);
\draw[->] (0,2) arc (90:270:.5);
\draw[<-] (6,0) arc (-90:90:.5);
\draw[<-] (0,2) -- (2,2);
\draw[->] (2,0) -- (3,0);
\draw[->] (3,2) -- (4,2);
\draw[<-] (4,0) -- (6,0);
\foreach \x/ \y in {0/0, 1/0}
    {
    \draw[<-] (\x+1,\y) -- (\x,\y+1);
    \draw[color=white, line width=10] (\x,\y) -- (\x+1,\y+1);
    \draw[->] (\x,\y) -- (\x+1,\y+1);
    }
\foreach \x/ \y in {3/0}
    {
    \draw[->] (\x+1,\y) -- (\x,\y+1);
    \draw[color=white, line width=10] (\x,\y) -- (\x+1,\y+1);
    \draw[->] (\x,\y) -- (\x+1,\y+1);
    }
\foreach \x/ \y in {2/1}
    {
    \draw[->] (\x,\y) -- (\x+1,\y+1);
    \draw[color=white, line width=10] (\x+1,\y) -- (\x,\y+1);
    \draw[->] (\x+1,\y) -- (\x,\y+1);
    }
\foreach \x/ \y in {4/1, 5/1}
    {
    \draw[->] (\x,\y) -- (\x+1,\y+1);
    \draw[color=white, line width=10] (\x+1,\y) -- (\x,\y+1);
    \draw[<-] (\x+1,\y) -- (\x,\y+1);
    }
\foreach \x/ \y in {3/0}
	{
\filldraw[shift={(\x+.5,\y+.5)}] [fill=black, draw=black, rounded corners] (0,0) circle (5pt);
	}
\foreach \x/ \y in {2/1}
	{
\filldraw[shift={(\x+.5,\y+.5)}] [fill=lightgray, draw=lightgray, rounded corners] (0,0) circle (5pt);
	}
\end{tikzpicture}		&
\begin{tikzpicture}[scale=.5]
\draw[->] (-1,0) -- (0,0);
\draw[->] (6,2) -- (7,2);
\draw[->] (0,2) arc (90:270:.5);
\draw[<-] (6,0) arc (-90:90:.5);
\draw[<-] (0,2) -- (2,2);
\draw[->] (2,0) -- (3,0);
\draw[->] (3,2) -- (4,2);
\draw[<-] (4,0) -- (6,0);
\foreach \x/ \y in {0/0, 1/0, 4/1, 5/1}
    {
    \draw[->] (\x,\y) .. controls (\x+.33,\y+.33) and (\x+.66,\y+.33) .. (\x+1,\y);
    \draw[->] (\x,\y+1) .. controls (\x+.33,\y+.66) and (\x+.66,\y+.66) .. (\x+1,\y+1);
    }
\foreach \x/ \y in {2/1, 3/0}
    {
    \draw[->] (\x,\y) .. controls (\x+.33,\y+.33) and (\x+.33,\y+.66) .. (\x,\y+1);
    \draw[->] (\x+1,\y) .. controls (\x+.66,\y+.33) and (\x+.66,\y+.66) .. (\x+1,\y+1);
    }
\foreach \x/ \y in {3/0}
	{
\draw[black, ultra thick] (\x+1,\y) .. controls (\x+.66,\y+.33) and (\x+.66,\y+.66) .. (\x+1,\y+1);
	}
\foreach \x/ \y in {2/1}
	{
\draw[lightgray, ultra thick] (\x+1,\y) .. controls (\x+.66,\y+.33) and (\x+.66,\y+.66) .. (\x+1,\y+1);
	}
\draw[black, ultra thick] (-1,0) -- (0,0);
\draw[black, ultra thick] (0,2) arc (90:270:.5);
\end{tikzpicture}		& $6_3$ \\
\hline
\end{tabular}
\caption{All reduced billiard table words for $c=6$ crossings with their associated alternating words, alternating diagrams, and Seifert circles.  The viable vertically-smoothed crossings of the alternating diagram are marked in black; these correspond to new Seifert circles.  The non-viable ones marked in gray.}
\label{tab:c6}
\end{table}


\begin{example}
\label{ex:c7}
Let us consider the set of all knots obtained in this model for crossing number $c=7$.  Since $7\equiv 1$ modulo 3, we only consider adding $d\equiv 3$ modulo 3 doubles, giving us the data in Table \ref{tab:c7}.

The actual average number $s$ of Seifert circles for $c=7$ is $2+\frac{26}{11}=\frac{48}{11}=4\frac{4}{11}$.

The upper bound on the average number of Seifert circles given by Theorem \ref{thm:Seifert} is $2+\frac{32}{11}=\frac{54}{11}=4\frac{10}{11}$.

By the genus formula for alternating knots, the actual average genus is $g(K)=1-\frac{1+s-c}{2}=\frac{1-s+c}{2}=\frac{1+c}{2}-\frac{s}{2}=4-\frac{24}{11}=\frac{20}{11}=1\frac{9}{11}$.

The lower bound on the average genus is $g(K)=4-\frac{27}{11}=\frac{17}{11}=1\frac{6}{11}$.
\end{example}


\begin{table}
\begin{tabular}{|l|c|c|c|c|}
\hline
Billiard Table &	Alternating Word 																														& Diagram $D$ & Seifert circles & Knot \\
\hspace{5mm} Word	$w$	 & & &  & \hspace{1mm} $K$ \\
\hline
+${\lightgray\Minus}\Plus$-${\lightgray\Plus}\Minus$+ 				&	$\sigma_1\sigma_2^{-1}\sigma_1\sigma_2^{-1}\sigma_1\sigma_2^{-1}\sigma_1$		& 
\begin{tikzpicture}[scale=.5]
\draw[->] (-1,0) -- (0,0);
\draw[->] (7,0) -- (8,0);
\draw[<-] (0,2) -- (1,2);
\draw[->] (2,2) -- (3,2);
\draw[->] (4,2) -- (5,2);
\draw[<-] (6,2) -- (7,2);
\draw[->] (1,0) -- (2,0);
\draw[<-] (3,0) -- (4,0);
\draw[->] (5,0) -- (6,0);
\draw[->] (0,2) arc (90:270:.5);
\draw[->] (7,1) arc (-90:90:.5);
\foreach \x/ \y in {2/0}
    {
    \draw[->] (\x+1,\y) -- (\x,\y+1);
    \draw[color=white, line width=10] (\x,\y) -- (\x+1,\y+1);
    \draw[->] (\x,\y) -- (\x+1,\y+1);
    }
\foreach \x/ \y in {0/0, 6/0}
    {
    \draw[<-] (\x+1,\y) -- (\x,\y+1);
    \draw[color=white, line width=10] (\x,\y) -- (\x+1,\y+1);
    \draw[->] (\x,\y) -- (\x+1,\y+1);
    }
\foreach \x/ \y in {4/0}
    {
    \draw[<-] (\x+1,\y) -- (\x,\y+1);
    \draw[color=white, line width=10] (\x,\y) -- (\x+1,\y+1);
    \draw[<-] (\x,\y) -- (\x+1,\y+1);
    }
\foreach \x/ \y in {1/1}
    {
    \draw[->] (\x,\y) -- (\x+1,\y+1);
    \draw[color=white, line width=10] (\x+1,\y) -- (\x,\y+1);
    \draw[->] (\x+1,\y) -- (\x,\y+1);
    }
\foreach \x/ \y in {3/1}
    {
    \draw[->] (\x,\y) -- (\x+1,\y+1);
    \draw[color=white, line width=10] (\x+1,\y) -- (\x,\y+1);
    \draw[<-] (\x+1,\y) -- (\x,\y+1);
    }
\foreach \x/ \y in {5/1}
    {
    \draw[<-] (\x,\y) -- (\x+1,\y+1);
    \draw[color=white, line width=10] (\x+1,\y) -- (\x,\y+1);
    \draw[<-] (\x+1,\y) -- (\x,\y+1);
    }
\foreach \x/ \y in {2/0, 5/1}
	{
\filldraw[shift={(\x+.5,\y+.5)}] [fill=black, draw=black, rounded corners] (0,0) circle (5pt);
	}
\foreach \x/ \y in {1/1, 4/0}
	{
\filldraw[shift={(\x+.5,\y+.5)}] [fill=lightgray, draw=lightgray, rounded corners] (0,0) circle (5pt);
	}
\end{tikzpicture}			&
\begin{tikzpicture}[scale=.5]
\draw[->] (-1,0) -- (0,0);
\draw[->] (7,0) -- (8,0);
\draw[<-] (0,2) -- (1,2);
\draw[->] (2,2) -- (3,2);
\draw[->] (4,2) -- (5,2);
\draw[<-] (6,2) -- (7,2);
\draw[->] (1,0) -- (2,0);
\draw[<-] (3,0) -- (4,0);
\draw[->] (5,0) -- (6,0);
\draw[->] (0,2) arc (90:270:.5);
\draw[->] (7,1) arc (-90:90:.5);
\foreach \x/ \y in {2/0}
    {
    \draw[->] (\x,\y) .. controls (\x+.33,\y+.33) and (\x+.33,\y+.66) .. (\x,\y+1);
    \draw[->] (\x+1,\y) .. controls (\x+.66,\y+.33) and (\x+.66,\y+.66) .. (\x+1,\y+1);
    }
\foreach \x/ \y in {0/0, 6/0}
    {
    \draw[->] (\x,\y) .. controls (\x+.33,\y+.33) and (\x+.66,\y+.33) .. (\x+1,\y);
    \draw[->] (\x,\y+1) .. controls (\x+.33,\y+.66) and (\x+.66,\y+.66) .. (\x+1,\y+1);
    }
\foreach \x/ \y in {4/0}
    {
    \draw[<-] (\x,\y) .. controls (\x+.33,\y+.33) and (\x+.33,\y+.66) .. (\x,\y+1);
    \draw[<-] (\x+1,\y) .. controls (\x+.66,\y+.33) and (\x+.66,\y+.66) .. (\x+1,\y+1);
    }
\foreach \x/ \y in {1/1}
    {
    \draw[->] (\x,\y) .. controls (\x+.33,\y+.33) and (\x+.33,\y+.66) .. (\x,\y+1);
    \draw[->] (\x+1,\y) .. controls (\x+.66,\y+.33) and (\x+.66,\y+.66) .. (\x+1,\y+1);
    }
\foreach \x/ \y in {3/1}
    {
    \draw[->] (\x,\y) .. controls (\x+.33,\y+.33) and (\x+.66,\y+.33) .. (\x+1,\y);
    \draw[->] (\x,\y+1) .. controls (\x+.33,\y+.66) and (\x+.66,\y+.66) .. (\x+1,\y+1);
    }
\foreach \x/ \y in {5/1}
    {
    \draw[<-] (\x,\y) .. controls (\x+.33,\y+.33) and (\x+.33,\y+.66) .. (\x,\y+1);
    \draw[<-] (\x+1,\y) .. controls (\x+.66,\y+.33) and (\x+.66,\y+.66) .. (\x+1,\y+1);
    }
\foreach \x/ \y in {2/0, 5/1}
	{
\draw[black, ultra thick] (\x,\y) .. controls (\x+.33,\y+.33) and (\x+.33,\y+.66) .. (\x,\y+1);
	}
\foreach \x/ \y in {1/1, 4/0}
	{
\draw[lightgray, ultra thick] (\x+1,\y) .. controls (\x+.66,\y+.33) and (\x+.66,\y+.66) .. (\x+1,\y+1);
	}
\draw[black, ultra thick] (-1,0) -- (0,0);
\draw[black, ultra thick] (0,2) arc (90:270:.5);
\end{tikzpicture}			& $7_7$   \\
\hline
+${\lightgray\Minus}\Plus{\lightgray\Minus\Minus}\Plus\Plus$-{}-+	& $\sigma_1\sigma_2^{-1}\sigma_1^2\sigma_2^{-1}\sigma_1^2$					& 
\begin{tikzpicture}[scale=.5]
\draw[->] (-1,0) -- (0,0);
\draw[->] (7,0) -- (8,0);
\draw[<-] (0,2) -- (1,2);
\draw[->] (2,2) -- (3,2);
\draw[->] (3,2) -- (4,2);
\draw[->] (4,0) -- (5,0);
\draw[<-] (5,2) -- (6,2);
\draw[<-] (6,2) -- (7,2);
\draw[->] (1,0) -- (2,0);
\draw[->] (0,2) arc (90:270:.5);
\draw[->] (7,1) arc (-90:90:.5);
\foreach \x/ \y in {2/0}
    {
    \draw[->] (\x+1,\y) -- (\x,\y+1);
    \draw[color=white, line width=10] (\x,\y) -- (\x+1,\y+1);
    \draw[->] (\x,\y) -- (\x+1,\y+1);
    }
\foreach \x/ \y in {0/0, 5/0, 6/0}
    {
    \draw[<-] (\x+1,\y) -- (\x,\y+1);
    \draw[color=white, line width=10] (\x,\y) -- (\x+1,\y+1);
    \draw[->] (\x,\y) -- (\x+1,\y+1);
    }
\foreach \x/ \y in {3/0}
    {
    \draw[<-] (\x+1,\y) -- (\x,\y+1);
    \draw[color=white, line width=10] (\x,\y) -- (\x+1,\y+1);
    \draw[<-] (\x,\y) -- (\x+1,\y+1);
    }
\foreach \x/ \y in {1/1}
    {
    \draw[->] (\x,\y) -- (\x+1,\y+1);
    \draw[color=white, line width=10] (\x+1,\y) -- (\x,\y+1);
    \draw[->] (\x+1,\y) -- (\x,\y+1);
    }
\foreach \x/ \y in {4/1}
    {
    \draw[<-] (\x,\y) -- (\x+1,\y+1);
    \draw[color=white, line width=10] (\x+1,\y) -- (\x,\y+1);
    \draw[<-] (\x+1,\y) -- (\x,\y+1);
    }
\foreach \x/ \y in {2/0, 4/1}
	{
\filldraw[shift={(\x+.5,\y+.5)}] [fill=black, draw=black, rounded corners] (0,0) circle (5pt);
	}
\foreach \x/ \y in {1/1, 3/0}
	{
\filldraw[shift={(\x+.5,\y+.5)}] [fill=lightgray, draw=lightgray, rounded corners] (0,0) circle (5pt);
	}
\end{tikzpicture}		&
\begin{tikzpicture}[scale=.5]
\draw[->] (-1,0) -- (0,0);
\draw[->] (7,0) -- (8,0);
\draw[<-] (0,2) -- (1,2);
\draw[->] (2,2) -- (3,2);
\draw[->] (3,2) -- (4,2);
\draw[->] (4,0) -- (5,0);
\draw[<-] (5,2) -- (6,2);
\draw[<-] (6,2) -- (7,2);
\draw[->] (1,0) -- (2,0);
\draw[->] (0,2) arc (90:270:.5);
\draw[->] (7,1) arc (-90:90:.5);
\foreach \x/ \y in {2/0}
    {
    \draw[->] (\x,\y) .. controls (\x+.33,\y+.33) and (\x+.33,\y+.66) .. (\x,\y+1);
    \draw[->] (\x+1,\y) .. controls (\x+.66,\y+.33) and (\x+.66,\y+.66) .. (\x+1,\y+1);
    }
\foreach \x/ \y in {0/0, 5/0, 6/0}
    {
    \draw[->] (\x,\y) .. controls (\x+.33,\y+.33) and (\x+.66,\y+.33) .. (\x+1,\y);
    \draw[->] (\x,\y+1) .. controls (\x+.33,\y+.66) and (\x+.66,\y+.66) .. (\x+1,\y+1);
    }
\foreach \x/ \y in {3/0}
    {
    \draw[<-] (\x,\y) .. controls (\x+.33,\y+.33) and (\x+.33,\y+.66) .. (\x,\y+1);
    \draw[<-] (\x+1,\y) .. controls (\x+.66,\y+.33) and (\x+.66,\y+.66) .. (\x+1,\y+1);
    }
\foreach \x/ \y in {1/1}
    {
    \draw[->] (\x,\y) .. controls (\x+.33,\y+.33) and (\x+.33,\y+.66) .. (\x,\y+1);
    \draw[->] (\x+1,\y) .. controls (\x+.66,\y+.33) and (\x+.66,\y+.66) .. (\x+1,\y+1);
    }
\foreach \x/ \y in {4/1}
    {
    \draw[<-] (\x,\y) .. controls (\x+.33,\y+.33) and (\x+.33,\y+.66) .. (\x,\y+1);
    \draw[<-] (\x+1,\y) .. controls (\x+.66,\y+.33) and (\x+.66,\y+.66) .. (\x+1,\y+1);
    }
\foreach \x/ \y in {2/0, 4/1}
	{
\draw[black, ultra thick] (\x+1,\y) .. controls (\x+.66,\y+.33) and (\x+.66,\y+.66) .. (\x+1,\y+1);
	}
\foreach \x/ \y in {1/1, 3/0}
	{
\draw[lightgray, ultra thick] (\x+1,\y) .. controls (\x+.66,\y+.33) and (\x+.66,\y+.66) .. (\x+1,\y+1);
	}
\draw[black, ultra thick] (-1,0) -- (0,0);
\draw[black, ultra thick] (0,2) arc (90:270:.5);
\end{tikzpicture}		& $7_6$ \\

\hline
+$\Minus\Plus\Plus\Minus\Plus\Plus$-{}-+	&	$\sigma_1\sigma_2^{-4}\sigma_1^2$							& 
\begin{tikzpicture}[scale=.5]
\draw[->] (-1,0) -- (0,0);
\draw[->] (7,0) -- (8,0);
\draw[<-] (0,2) -- (1,2);
\draw[->] (1,0) -- (2,0);
\draw[->] (2,0) -- (3,0);
\draw[->] (3,0) -- (4,0);
\draw[->] (4,0) -- (5,0);
\draw[<-] (5,2) -- (6,2);
\draw[<-] (6,2) -- (7,2);
\draw[->] (0,2) arc (90:270:.5);
\draw[->] (7,1) arc (-90:90:.5);
\foreach \x/ \y in {0/0, 5/0, 6/0}
    {
    \draw[<-] (\x+1,\y) -- (\x,\y+1);
    \draw[color=white, line width=10] (\x,\y) -- (\x+1,\y+1);
    \draw[->] (\x,\y) -- (\x+1,\y+1);
    }
\foreach \x/ \y in {1/1,3/1}
    {
    \draw[->] (\x,\y) -- (\x+1,\y+1);
    \draw[color=white, line width=10] (\x+1,\y) -- (\x,\y+1);
    \draw[->] (\x+1,\y) -- (\x,\y+1);
    }
\foreach \x/ \y in {2/1,4/1}
    {
    \draw[<-] (\x,\y) -- (\x+1,\y+1);
    \draw[color=white, line width=10] (\x+1,\y) -- (\x,\y+1);
    \draw[<-] (\x+1,\y) -- (\x,\y+1);
    }
\foreach \x/ \y in {1/1, 2/1, 3/1, 4/1}
	{
\filldraw[shift={(\x+.5,\y+.5)}] [fill=black, draw=black, rounded corners] (0,0) circle (5pt);
	}
\end{tikzpicture}		&
\begin{tikzpicture}[scale=.5]
\draw[->] (-1,0) -- (0,0);
\draw[->] (7,0) -- (8,0);
\draw[<-] (0,2) -- (1,2);
\draw[->] (1,0) -- (2,0);
\draw[->] (2,0) -- (3,0);
\draw[->] (3,0) -- (4,0);
\draw[->] (4,0) -- (5,0);
\draw[<-] (5,2) -- (6,2);
\draw[<-] (6,2) -- (7,2);
\draw[->] (0,2) arc (90:270:.5);
\draw[->] (7,1) arc (-90:90:.5);
\foreach \x/ \y in {0/0, 5/0, 6/0}
    {
    \draw[->] (\x,\y) .. controls (\x+.33,\y+.33) and (\x+.66,\y+.33) .. (\x+1,\y);
    \draw[->] (\x,\y+1) .. controls (\x+.33,\y+.66) and (\x+.66,\y+.66) .. (\x+1,\y+1);
    }
\foreach \x/ \y in {1/1,3/1}
    {
    \draw[->] (\x,\y) .. controls (\x+.33,\y+.33) and (\x+.33,\y+.66) .. (\x,\y+1);
    \draw[->] (\x+1,\y) .. controls (\x+.66,\y+.33) and (\x+.66,\y+.66) .. (\x+1,\y+1);
    }
\foreach \x/ \y in {2/1,4/1}
    {
    \draw[<-] (\x,\y) .. controls (\x+.33,\y+.33) and (\x+.33,\y+.66) .. (\x,\y+1);
    \draw[<-] (\x+1,\y) .. controls (\x+.66,\y+.33) and (\x+.66,\y+.66) .. (\x+1,\y+1);
    }
\foreach \x/ \y in {1/1, 2/1, 3/1, 4/1}
	{
\draw[black, ultra thick] (\x+1,\y) .. controls (\x+.66,\y+.33) and (\x+.66,\y+.66) .. (\x+1,\y+1);
	}
\draw[black, ultra thick] (-1,0) -- (0,0);
\draw[black, ultra thick] (0,2) arc (90:270:.5);
\end{tikzpicture}		& $7_2$ \\
\hline
+$\Minus\Plus\Plus$-{}-+-{}-+	&	$\sigma_1\sigma_2^{-2}\sigma_1^4$												& 
\begin{tikzpicture}[scale=.5]
\draw[->] (-1,0) -- (0,0);
\draw[->] (7,0) -- (8,0);
\draw[<-] (0,2) -- (1,2);
\draw[->] (1,0) -- (2,0);
\draw[->] (2,0) -- (3,0);
\draw[<-] (3,2) -- (4,2);
\draw[<-] (4,2) -- (5,2);
\draw[<-] (5,2) -- (6,2);
\draw[<-] (6,2) -- (7,2);
\draw[->] (0,2) arc (90:270:.5);
\draw[->] (7,1) arc (-90:90:.5);
\foreach \x/ \y in {0/0, 3/0, 4/0, 5/0, 6/0}
    {
    \draw[<-] (\x+1,\y) -- (\x,\y+1);
    \draw[color=white, line width=10] (\x,\y) -- (\x+1,\y+1);
    \draw[->] (\x,\y) -- (\x+1,\y+1);
    }
\foreach \x/ \y in {1/1}
    {
    \draw[->] (\x,\y) -- (\x+1,\y+1);
    \draw[color=white, line width=10] (\x+1,\y) -- (\x,\y+1);
    \draw[->] (\x+1,\y) -- (\x,\y+1);
    }
\foreach \x/ \y in {2/1}
    {
    \draw[<-] (\x,\y) -- (\x+1,\y+1);
    \draw[color=white, line width=10] (\x+1,\y) -- (\x,\y+1);
    \draw[<-] (\x+1,\y) -- (\x,\y+1);
    }
\foreach \x/ \y in {1/1, 2/1}
	{
\filldraw[shift={(\x+.5,\y+.5)}] [fill=black, draw=black, rounded corners] (0,0) circle (5pt);
	}
\end{tikzpicture}		&
\begin{tikzpicture}[scale=.5]
\draw[->] (-1,0) -- (0,0);
\draw[->] (7,0) -- (8,0);
\draw[<-] (0,2) -- (1,2);
\draw[->] (1,0) -- (2,0);
\draw[->] (2,0) -- (3,0);
\draw[<-] (3,2) -- (4,2);
\draw[<-] (4,2) -- (5,2);
\draw[<-] (5,2) -- (6,2);
\draw[<-] (6,2) -- (7,2);
\draw[->] (0,2) arc (90:270:.5);
\draw[->] (7,1) arc (-90:90:.5);
\foreach \x/ \y in {0/0, 3/0, 4/0, 5/0, 6/0}
    {
    \draw[->] (\x,\y) .. controls (\x+.33,\y+.33) and (\x+.66,\y+.33) .. (\x+1,\y);
    \draw[->] (\x,\y+1) .. controls (\x+.33,\y+.66) and (\x+.66,\y+.66) .. (\x+1,\y+1);
    }
\foreach \x/ \y in {1/1}
    {
    \draw[->] (\x,\y) .. controls (\x+.33,\y+.33) and (\x+.33,\y+.66) .. (\x,\y+1);
    \draw[->] (\x+1,\y) .. controls (\x+.66,\y+.33) and (\x+.66,\y+.66) .. (\x+1,\y+1);
    }
\foreach \x/ \y in {2/1}
    {
    \draw[<-] (\x,\y) .. controls (\x+.33,\y+.33) and (\x+.33,\y+.66) .. (\x,\y+1);
    \draw[<-] (\x+1,\y) .. controls (\x+.66,\y+.33) and (\x+.66,\y+.66) .. (\x+1,\y+1);
    }
\foreach \x/ \y in {1/1, 2/1}
	{
\draw[black, ultra thick] (\x+1,\y) .. controls (\x+.66,\y+.33) and (\x+.66,\y+.66) .. (\x+1,\y+1);
	}
\draw[black, ultra thick] (-1,0) -- (0,0);
\draw[black, ultra thick] (0,2) arc (90:270:.5);
\end{tikzpicture}		& $7_3$ \\
\hline
+$\Minus\Plus\Plus$-{}-$\Plus\Plus\Minus$+	&	$\sigma_1\sigma_2^{-2}\sigma_1\sigma_2^{-2}\sigma_1$						& 
\begin{tikzpicture}[scale=.5]
\draw[->] (-1,0) -- (0,0);
\draw[->] (7,0) -- (8,0);
\draw[<-] (0,2) -- (1,2);
\draw[->] (1,0) -- (2,0);
\draw[->] (2,0) -- (3,0);
\draw[<-] (3,2) -- (4,2);
\draw[->] (4,0) -- (5,0);
\draw[->] (5,0) -- (6,0);
\draw[<-] (6,2) -- (7,2);
\draw[->] (0,2) arc (90:270:.5);
\draw[->] (7,1) arc (-90:90:.5);
\foreach \x/ \y in {0/0, 3/0, 6/0}
    {
    \draw[<-] (\x+1,\y) -- (\x,\y+1);
    \draw[color=white, line width=10] (\x,\y) -- (\x+1,\y+1);
    \draw[->] (\x,\y) -- (\x+1,\y+1);
    }
\foreach \x/ \y in {1/1, 4/1}
    {
    \draw[->] (\x,\y) -- (\x+1,\y+1);
    \draw[color=white, line width=10] (\x+1,\y) -- (\x,\y+1);
    \draw[->] (\x+1,\y) -- (\x,\y+1);
    }
\foreach \x/ \y in {2/1, 5/1}
    {
    \draw[<-] (\x,\y) -- (\x+1,\y+1);
    \draw[color=white, line width=10] (\x+1,\y) -- (\x,\y+1);
    \draw[<-] (\x+1,\y) -- (\x,\y+1);
    }
\foreach \x/ \y in {1/1, 2/1, 4/1, 5/1}
	{
\filldraw[shift={(\x+.5,\y+.5)}] [fill=black, draw=black, rounded corners] (0,0) circle (5pt);
	}
\end{tikzpicture}		&
\begin{tikzpicture}[scale=.5]
\draw[->] (-1,0) -- (0,0);
\draw[->] (7,0) -- (8,0);
\draw[<-] (0,2) -- (1,2);
\draw[->] (1,0) -- (2,0);
\draw[->] (2,0) -- (3,0);
\draw[<-] (3,2) -- (4,2);
\draw[->] (4,0) -- (5,0);
\draw[->] (5,0) -- (6,0);
\draw[<-] (6,2) -- (7,2);
\draw[->] (0,2) arc (90:270:.5);
\draw[->] (7,1) arc (-90:90:.5);
\foreach \x/ \y in {0/0, 3/0, 6/0}
    {
    \draw[->] (\x,\y) .. controls (\x+.33,\y+.33) and (\x+.66,\y+.33) .. (\x+1,\y);
    \draw[->] (\x,\y+1) .. controls (\x+.33,\y+.66) and (\x+.66,\y+.66) .. (\x+1,\y+1);
    }
\foreach \x/ \y in {1/1, 4/1}
    {
    \draw[->] (\x,\y) .. controls (\x+.33,\y+.33) and (\x+.33,\y+.66) .. (\x,\y+1);
    \draw[->] (\x+1,\y) .. controls (\x+.66,\y+.33) and (\x+.66,\y+.66) .. (\x+1,\y+1);
    }
\foreach \x/ \y in {2/1, 5/1}
    {
    \draw[<-] (\x,\y) .. controls (\x+.33,\y+.33) and (\x+.33,\y+.66) .. (\x,\y+1);
    \draw[<-] (\x+1,\y) .. controls (\x+.66,\y+.33) and (\x+.66,\y+.66) .. (\x+1,\y+1);
    }
\foreach \x/ \y in {1/1, 2/1, 4/1, 5/1}
	{
\draw[black, ultra thick] (\x+1,\y) .. controls (\x+.66,\y+.33) and (\x+.66,\y+.66) .. (\x+1,\y+1);
	}
\draw[black, ultra thick] (-1,0) -- (0,0);
\draw[black, ultra thick] (0,2) arc (90:270:.5);
\end{tikzpicture}		& $7_4$ \\
\hline
+-{}-+$\Minus\Plus\Plus$-{}-+	&	$\sigma_1^3\sigma_2^{-2}\sigma_1^2$											&
\begin{tikzpicture}[scale=.5]
\draw[->] (-1,0) -- (0,0);
\draw[->] (7,0) -- (8,0);
\draw[<-] (0,2) -- (1,2);
\draw[<-] (1,2) -- (2,2);
\draw[<-] (2,2) -- (3,2);
\draw[->] (3,0) -- (4,0);
\draw[->] (4,0) -- (5,0);
\draw[<-] (5,2) -- (6,2);
\draw[<-] (6,2) -- (7,2);
\draw[->] (0,2) arc (90:270:.5);
\draw[->] (7,1) arc (-90:90:.5);
\foreach \x/ \y in {0/0, 1/0, 2/0, 5/0, 6/0}
    {
    \draw[<-] (\x+1,\y) -- (\x,\y+1);
    \draw[color=white, line width=10] (\x,\y) -- (\x+1,\y+1);
    \draw[->] (\x,\y) -- (\x+1,\y+1);
    }
\foreach \x/ \y in {3/1}
    {
    \draw[->] (\x,\y) -- (\x+1,\y+1);
    \draw[color=white, line width=10] (\x+1,\y) -- (\x,\y+1);
    \draw[->] (\x+1,\y) -- (\x,\y+1);
    }
\foreach \x/ \y in {4/1}
    {
    \draw[<-] (\x,\y) -- (\x+1,\y+1);
    \draw[color=white, line width=10] (\x+1,\y) -- (\x,\y+1);
    \draw[<-] (\x+1,\y) -- (\x,\y+1);
    }
\foreach \x/ \y in {3/1, 4/1}
	{
\filldraw[shift={(\x+.5,\y+.5)}] [fill=black, draw=black, rounded corners] (0,0) circle (5pt);
	}
\end{tikzpicture}			&
\begin{tikzpicture}[scale=.5]
\draw[->] (-1,0) -- (0,0);
\draw[->] (7,0) -- (8,0);
\draw[<-] (0,2) -- (1,2);
\draw[<-] (1,2) -- (2,2);
\draw[<-] (2,2) -- (3,2);
\draw[->] (3,0) -- (4,0);
\draw[->] (4,0) -- (5,0);
\draw[<-] (5,2) -- (6,2);
\draw[<-] (6,2) -- (7,2);
\draw[->] (0,2) arc (90:270:.5);
\draw[->] (7,1) arc (-90:90:.5);
\foreach \x/ \y in {0/0, 1/0, 2/0, 5/0, 6/0}
    {
    \draw[->] (\x,\y) .. controls (\x+.33,\y+.33) and (\x+.66,\y+.33) .. (\x+1,\y);
    \draw[->] (\x,\y+1) .. controls (\x+.33,\y+.66) and (\x+.66,\y+.66) .. (\x+1,\y+1);
    }
\foreach \x/ \y in {3/1}
    {
    \draw[->] (\x,\y) .. controls (\x+.33,\y+.33) and (\x+.33,\y+.66) .. (\x,\y+1);
    \draw[->] (\x+1,\y) .. controls (\x+.66,\y+.33) and (\x+.66,\y+.66) .. (\x+1,\y+1);
    }
\foreach \x/ \y in {4/1}
    {
    \draw[<-] (\x,\y) .. controls (\x+.33,\y+.33) and (\x+.33,\y+.66) .. (\x,\y+1);
    \draw[<-] (\x+1,\y) .. controls (\x+.66,\y+.33) and (\x+.66,\y+.66) .. (\x+1,\y+1);
    }
\foreach \x/ \y in {3/1, 4/1}
	{
\draw[black, ultra thick] (\x+1,\y) .. controls (\x+.66,\y+.33) and (\x+.66,\y+.66) .. (\x+1,\y+1);
	}
\draw[black, ultra thick] (-1,0) -- (0,0);
\draw[black, ultra thick] (0,2) arc (90:270:.5);
\end{tikzpicture}			& $7_5$ \\

\hline
+-{}-+-{}-+-{}-+	&	$\sigma_1^7$																			&
\begin{tikzpicture}[scale=.5]
\draw[->] (-1,0) -- (0,0);
\draw[->] (7,0) -- (8,0);
\draw[<-] (0,2) -- (7,2);
\draw[->] (0,2) arc (90:270:.5);
\draw[->] (7,1) arc (-90:90:.5);
\draw[white] (1,2.25) -- (2,2.25);

\foreach \x/ \y in {0/0, 1/0, 2/0, 3/0, 4/0, 5/0, 6/0}
    {
    \draw[<-] (\x+1,\y) -- (\x,\y+1);
    \draw[color=white, line width=10] (\x,\y) -- (\x+1,\y+1);
    \draw[->] (\x,\y) -- (\x+1,\y+1);
		}
\end{tikzpicture}			&
\begin{tikzpicture}[scale=.5]
\draw[->] (-1,0) -- (0,0);
\draw[->] (7,0) -- (8,0);
\draw[<-] (0,2) -- (7,2);
\draw[->] (0,2) arc (90:270:.5);
\draw[->] (7,1) arc (-90:90:.5);
\draw[white] (1,2.25) -- (2,2.25);

\foreach \x/ \y in {0/0, 1/0, 2/0, 3/0, 4/0, 5/0, 6/0}
    {
    \draw[->] (\x,\y) .. controls (\x+.33,\y+.33) and (\x+.66,\y+.33) .. (\x+1,\y);
    \draw[->] (\x,\y+1) .. controls (\x+.33,\y+.66) and (\x+.66,\y+.66) .. (\x+1,\y+1);
    }
\draw[black, ultra thick] (-1,0) -- (0,0);
\draw[black, ultra thick] (0,2) arc (90:270:.5);
\end{tikzpicture}			& $7_1$  \\
\hline
+-{}-+-{}-$\Plus\Plus\Minus$+	&	$\sigma_1^4\sigma_2^{-2}\sigma_1$									& 
\begin{tikzpicture}[scale=.5]
\draw[->] (-1,0) -- (0,0);
\draw[->] (7,0) -- (8,0);
\draw[<-] (0,2) -- (4,2);
\draw[->] (4,0) -- (6,0);
\draw[<-] (6,2) -- (7,2);
\draw[->] (0,2) arc (90:270:.5);
\draw[->] (7,1) arc (-90:90:.5);

\foreach \x/ \y in {0/0, 1/0, 2/0, 3/0, 6/0}
    {
    \draw[<-] (\x+1,\y) -- (\x,\y+1);
    \draw[color=white, line width=10] (\x,\y) -- (\x+1,\y+1);
    \draw[->] (\x,\y) -- (\x+1,\y+1);
		}
\foreach \x/ \y in {4/1}
    {
    \draw[->] (\x,\y) -- (\x+1,\y+1);
    \draw[color=white, line width=10] (\x+1,\y) -- (\x,\y+1);
    \draw[->] (\x+1,\y) -- (\x,\y+1);
    }
\foreach \x/ \y in {5/1}
    {
    \draw[<-] (\x,\y) -- (\x+1,\y+1);
    \draw[color=white, line width=10] (\x+1,\y) -- (\x,\y+1);
    \draw[<-] (\x+1,\y) -- (\x,\y+1);
    }
\foreach \x/ \y in {4/1, 5/1}
	{
\filldraw[shift={(\x+.5,\y+.5)}] [fill=black, draw=black, rounded corners] (0,0) circle (5pt);
	}
\end{tikzpicture}			&
\begin{tikzpicture}[scale=.5]
\draw[->] (-1,0) -- (0,0);
\draw[->] (7,0) -- (8,0);
\draw[<-] (0,2) -- (4,2);
\draw[->] (4,0) -- (6,0);
\draw[<-] (6,2) -- (7,2);
\draw[->] (0,2) arc (90:270:.5);
\draw[->] (7,1) arc (-90:90:.5);

\foreach \x/ \y in {0/0, 1/0, 2/0, 3/0, 6/0}
    {
    \draw[->] (\x,\y) .. controls (\x+.33,\y+.33) and (\x+.66,\y+.33) .. (\x+1,\y);
    \draw[->] (\x,\y+1) .. controls (\x+.33,\y+.66) and (\x+.66,\y+.66) .. (\x+1,\y+1);
    }
\foreach \x/ \y in {4/1}
    {
    \draw[->] (\x,\y) .. controls (\x+.33,\y+.33) and (\x+.33,\y+.66) .. (\x,\y+1);
    \draw[->] (\x+1,\y) .. controls (\x+.66,\y+.33) and (\x+.66,\y+.66) .. (\x+1,\y+1);
    }
\foreach \x/ \y in {5/1}
    {
    \draw[<-] (\x,\y) .. controls (\x+.33,\y+.33) and (\x+.33,\y+.66) .. (\x,\y+1);
    \draw[<-] (\x+1,\y) .. controls (\x+.66,\y+.33) and (\x+.66,\y+.66) .. (\x+1,\y+1);
    }
\foreach \x/ \y in {4/1, 5/1}
	{
\draw[black, ultra thick] (\x+1,\y) .. controls (\x+.66,\y+.33) and (\x+.66,\y+.66) .. (\x+1,\y+1);
	}
\draw[black, ultra thick] (-1,0) -- (0,0);
\draw[black, ultra thick] (0,2) arc (90:270:.5);
\end{tikzpicture}				& $7_3$ \\
\hline
+-{}-$\Plus\Plus\Minus$+-{}-+	&	$\sigma_1^2\sigma_2^{-2}\sigma_1^3$																			&
\begin{tikzpicture}[scale=.5]
\draw[->] (-1,0) -- (0,0);
\draw[->] (7,0) -- (8,0);
\draw[<-] (0,2) -- (2,2);
\draw[->] (2,0) -- (4,0);
\draw[<-] (4,2) -- (7,2);
\draw[->] (0,2) arc (90:270:.5);
\draw[->] (7,1) arc (-90:90:.5);

\foreach \x/ \y in {0/0, 1/0, 4/0, 5/0, 6/0}
    {
    \draw[<-] (\x+1,\y) -- (\x,\y+1);
    \draw[color=white, line width=10] (\x,\y) -- (\x+1,\y+1);
    \draw[->] (\x,\y) -- (\x+1,\y+1);
		}
\foreach \x/ \y in {2/1}
    {
    \draw[->] (\x,\y) -- (\x+1,\y+1);
    \draw[color=white, line width=10] (\x+1,\y) -- (\x,\y+1);
    \draw[->] (\x+1,\y) -- (\x,\y+1);
    }
\foreach \x/ \y in {3/1}
    {
    \draw[<-] (\x,\y) -- (\x+1,\y+1);
    \draw[color=white, line width=10] (\x+1,\y) -- (\x,\y+1);
    \draw[<-] (\x+1,\y) -- (\x,\y+1);
    }
\foreach \x/ \y in {2/1, 3/1}
	{
\filldraw[shift={(\x+.5,\y+.5)}] [fill=black, draw=black, rounded corners] (0,0) circle (5pt);
	}
\end{tikzpicture}			&
\begin{tikzpicture}[scale=.5]
\draw[->] (-1,0) -- (0,0);
\draw[->] (7,0) -- (8,0);
\draw[<-] (0,2) -- (2,2);
\draw[->] (2,0) -- (4,0);
\draw[<-] (4,2) -- (7,2);
\draw[->] (0,2) arc (90:270:.5);
\draw[->] (7,1) arc (-90:90:.5);

\foreach \x/ \y in {0/0, 1/0, 4/0, 5/0, 6/0}
    {
    \draw[->] (\x,\y) .. controls (\x+.33,\y+.33) and (\x+.66,\y+.33) .. (\x+1,\y);
    \draw[->] (\x,\y+1) .. controls (\x+.33,\y+.66) and (\x+.66,\y+.66) .. (\x+1,\y+1);
    }
\foreach \x/ \y in {2/1}
    {
    \draw[->] (\x,\y) .. controls (\x+.33,\y+.33) and (\x+.33,\y+.66) .. (\x,\y+1);
    \draw[->] (\x+1,\y) .. controls (\x+.66,\y+.33) and (\x+.66,\y+.66) .. (\x+1,\y+1);
    }
\foreach \x/ \y in {3/1}
    {
    \draw[<-] (\x,\y) .. controls (\x+.33,\y+.33) and (\x+.33,\y+.66) .. (\x,\y+1);
    \draw[<-] (\x+1,\y) .. controls (\x+.66,\y+.33) and (\x+.66,\y+.66) .. (\x+1,\y+1);
    }
\foreach \x/ \y in {2/1, 3/1}
	{
\draw[black, ultra thick] (\x+1,\y) .. controls (\x+.66,\y+.33) and (\x+.66,\y+.66) .. (\x+1,\y+1);
	}
\draw[black, ultra thick] (-1,0) -- (0,0);
\draw[black, ultra thick] (0,2) arc (90:270:.5);
\end{tikzpicture}					& $7_5$		\\
\hline
+-{}-$\Plus\Plus\Minus\Plus\Plus\Minus$+	&	$\sigma_1^2\sigma_2^{-4}\sigma_1$																				&				
\begin{tikzpicture}[scale=.5]
\draw[->] (-1,0) -- (0,0);
\draw[->] (7,0) -- (8,0);
\draw[<-] (0,2) -- (2,2);
\draw[->] (2,0) -- (6,0);
\draw[<-] (6,2) -- (7,2);
\draw[->] (0,2) arc (90:270:.5);
\draw[->] (7,1) arc (-90:90:.5);

\foreach \x/ \y in {0/0, 1/0, 6/0}
    {
    \draw[<-] (\x+1,\y) -- (\x,\y+1);
    \draw[color=white, line width=10] (\x,\y) -- (\x+1,\y+1);
    \draw[->] (\x,\y) -- (\x+1,\y+1);
		}
\foreach \x/ \y in {2/1,4/1}
    {
    \draw[->] (\x,\y) -- (\x+1,\y+1);
    \draw[color=white, line width=10] (\x+1,\y) -- (\x,\y+1);
    \draw[->] (\x+1,\y) -- (\x,\y+1);
    }
\foreach \x/ \y in {3/1,5/1}
    {
    \draw[<-] (\x,\y) -- (\x+1,\y+1);
    \draw[color=white, line width=10] (\x+1,\y) -- (\x,\y+1);
    \draw[<-] (\x+1,\y) -- (\x,\y+1);
    }
\foreach \x/ \y in {2/1, 3/1, 4/1, 5/1}
	{
\filldraw[shift={(\x+.5,\y+.5)}] [fill=black, draw=black, rounded corners] (0,0) circle (5pt);
	}
\end{tikzpicture}			&
\begin{tikzpicture}[scale=.5]
\draw[->] (-1,0) -- (0,0);
\draw[->] (7,0) -- (8,0);
\draw[<-] (0,2) -- (2,2);
\draw[->] (2,0) -- (6,0);
\draw[<-] (6,2) -- (7,2);
\draw[->] (0,2) arc (90:270:.5);
\draw[->] (7,1) arc (-90:90:.5);

\foreach \x/ \y in {0/0,1/0,6/0}
    {
    \draw[->] (\x,\y) .. controls (\x+.33,\y+.33) and (\x+.66,\y+.33) .. (\x+1,\y);
    \draw[->] (\x,\y+1) .. controls (\x+.33,\y+.66) and (\x+.66,\y+.66) .. (\x+1,\y+1);
    }
\foreach \x/ \y in {3/1, 5/1}
    {
    \draw[<-] (\x,\y) .. controls (\x+.33,\y+.33) and (\x+.33,\y+.66) .. (\x,\y+1);
    \draw[<-] (\x+1,\y) .. controls (\x+.66,\y+.33) and (\x+.66,\y+.66) .. (\x+1,\y+1);
    }
\foreach \x/ \y in {2/1, 4/1}
    {
    \draw[->] (\x,\y) .. controls (\x+.33,\y+.33) and (\x+.33,\y+.66) .. (\x,\y+1);
    \draw[->] (\x+1,\y) .. controls (\x+.66,\y+.33) and (\x+.66,\y+.66) .. (\x+1,\y+1);
    }
\foreach \x/ \y in {2/1, 3/1, 4/1, 5/1}
	{
\draw[black, ultra thick] (\x+1,\y) .. controls (\x+.66,\y+.33) and (\x+.66,\y+.66) .. (\x+1,\y+1);
	}
\draw[black, ultra thick] (-1,0) -- (0,0);
\draw[black, ultra thick] (0,2) arc (90:270:.5);
\end{tikzpicture} & $7_2$		\\
\hline
+-{}-${\lightgray\Plus\Plus}\Minus\Minus{\lightgray\Plus}\Minus$+	&	$\sigma_1^2\sigma_2^{-1}\sigma_1^2\sigma_2^{-1}\sigma_1$								&				
\begin{tikzpicture}[scale=.5]
\draw[->] (-1,0) -- (0,0);
\draw[->] (7,0) -- (8,0);
\draw[<-] (0,2) -- (2,2);
\draw[->] (2,0) -- (3,0);
\draw[->] (3,2) -- (5,2);
\draw[->] (5,0) -- (6,0);
\draw[<-] (6,2) -- (7,2);
\draw[->] (0,2) arc (90:270:.5);
\draw[->] (7,1) arc (-90:90:.5);

\foreach \x/ \y in {0/0, 1/0, 6/0}
    {
    \draw[<-] (\x+1,\y) -- (\x,\y+1);
    \draw[color=white, line width=10] (\x,\y) -- (\x+1,\y+1);
    \draw[->] (\x,\y) -- (\x+1,\y+1);
		}
\foreach \x/ \y in {3/0}
    {
    \draw[->] (\x+1,\y) -- (\x,\y+1);
    \draw[color=white, line width=10] (\x,\y) -- (\x+1,\y+1);
    \draw[->] (\x,\y) -- (\x+1,\y+1);
		}
\foreach \x/ \y in {4/0}
    {
    \draw[<-] (\x+1,\y) -- (\x,\y+1);
    \draw[color=white, line width=10] (\x,\y) -- (\x+1,\y+1);
    \draw[<-] (\x,\y) -- (\x+1,\y+1);
		}
\foreach \x/ \y in {2/1}
    {
    \draw[->] (\x,\y) -- (\x+1,\y+1);
    \draw[color=white, line width=10] (\x+1,\y) -- (\x,\y+1);
    \draw[->] (\x+1,\y) -- (\x,\y+1);
    }
\foreach \x/ \y in {5/1}
    {
    \draw[<-] (\x,\y) -- (\x+1,\y+1);
    \draw[color=white, line width=10] (\x+1,\y) -- (\x,\y+1);
    \draw[<-] (\x+1,\y) -- (\x,\y+1);
    }
\foreach \x/ \y in {3/0, 5/1}
	{
\filldraw[shift={(\x+.5,\y+.5)}] [fill=black, draw=black, rounded corners] (0,0) circle (5pt);
	}
\foreach \x/ \y in {2/1, 4/0}
	{
\filldraw[shift={(\x+.5,\y+.5)}] [fill=lightgray, draw=lightgray, rounded corners] (0,0) circle (5pt);
	}
\end{tikzpicture}			&
\begin{tikzpicture}[scale=.5]
\draw[->] (-1,0) -- (0,0);
\draw[->] (7,0) -- (8,0);
\draw[<-] (0,2) -- (2,2);
\draw[->] (2,0) -- (3,0);
\draw[->] (3,2) -- (5,2);
\draw[->] (5,0) -- (6,0);
\draw[<-] (6,2) -- (7,2);
\draw[->] (0,2) arc (90:270:.5);
\draw[->] (7,1) arc (-90:90:.5);

\foreach \x/ \y in {0/0,1/0,6/0}
    {
    \draw[->] (\x,\y) .. controls (\x+.33,\y+.33) and (\x+.66,\y+.33) .. (\x+1,\y);
    \draw[->] (\x,\y+1) .. controls (\x+.33,\y+.66) and (\x+.66,\y+.66) .. (\x+1,\y+1);
    }
\foreach \x/ \y in {4/0, 5/1}
    {
    \draw[<-] (\x,\y) .. controls (\x+.33,\y+.33) and (\x+.33,\y+.66) .. (\x,\y+1);
    \draw[<-] (\x+1,\y) .. controls (\x+.66,\y+.33) and (\x+.66,\y+.66) .. (\x+1,\y+1);
    }
\foreach \x/ \y in {2/1, 3/0}
    {
    \draw[->] (\x,\y) .. controls (\x+.33,\y+.33) and (\x+.33,\y+.66) .. (\x,\y+1);
    \draw[->] (\x+1,\y) .. controls (\x+.66,\y+.33) and (\x+.66,\y+.66) .. (\x+1,\y+1);
    }
\foreach \x/ \y in {3/0, 5/1}
	{
\draw[black, ultra thick] (\x+1,\y) .. controls (\x+.66,\y+.33) and (\x+.66,\y+.66) .. (\x+1,\y+1);
	}
\foreach \x/ \y in {2/1, 4/0}
	{
\draw[lightgray, ultra thick] (\x+1,\y) .. controls (\x+.66,\y+.33) and (\x+.66,\y+.66) .. (\x+1,\y+1);
	}
\draw[black, ultra thick] (-1,0) -- (0,0);
\draw[black, ultra thick] (0,2) arc (90:270:.5);
\end{tikzpicture}		& $7_6$		\\
\hline
\end{tabular}
\caption{All reduced billiard table words for $c=7$ crossings with their associated alternating words, alternating diagrams, and Seifert circles.  The viable vertically-smoothed crossings of the alternating diagram are marked in black; these correspond to new Seifert circles.  The non-viable ones marked in gray.}
\label{tab:c7}
\end{table}


\section{Counting contributions to genus via crossing index}

We now use the genus formula $g(K)=1-\frac{1+s-c}{2}$ for alternating knots and Theorem \ref{thm:Seifert} to calculate a lower bound for the genus of an average 2-bridge knot of given crossing number.  However, the following theorem performs the necessary summation \emph{not} by calculating the genus of each knot but by calculating the contribution to genus by each indexed crossing.

\begin{definition}
Recall from the proof of Theorem \ref{thm:number} that the reduced billiard table word $w$ has length $\ell=c+d$ where $c$ is the number of crossings in the alternating diagram $D$ and where $d$ is the number of doubles.

The crossing index $1 \leq i \leq c$ is the index of the crossing from left to right of the alternating diagram $D$.

Let $d_1(i)$ be the number of doubles appearing in the associated reduced billiard table word $w$ \emph{prior} to the crossing at index $i$ and let $d_2(i)$ be the number of doubles appearing \emph{after} the crossing at index $i$.

Then the sum $d_1(i)+d_2(i)$ is either equal to $d$ or one less than $d$, depending on whether or not the crossing at index in $D$ comes from a single in $w$, respectively.
\end{definition}

We use $d_1(i)$ and $d_2(i)$ below and also in the Main Theorem \ref{thm:main}.

\begin{definition}
For crossing index $1\leq i \leq c$ and number of doubles $d_1(i)$ prior to the crossing at index $i$, define $\delta(\{i\})$ and $\delta(\{i,i+1\})$ to be \textbf{vertical indicator functions} that are equal to 1 or 0 depending on whether the crossing in $D$ at index $i$ is smoothed vertically or horizontally, respectively, as follows:
\begin{align}
\delta(\{i\}) &= 
\begin{cases}
1 \text{ if } i+d_1(i)\not\equiv 1 \text{ mod } 3\\
0 \text{ if } i+d_1(i)\equiv 1 \text{ mod } 3
\end{cases} 
\text{ when the crossing at index $i$ arises from a single and }\\
\delta(\{i,i+1\}) &= 
\begin{cases}
1 \text{ if } i+d_1(i)\not\equiv 2 \text{ mod } 3\\
0 \text{ if } i+d_1(i)\equiv 2 \text{ mod } 3
\end{cases}
\text{ when the crossing at index $i$ arises from a double}.
\end{align}

Note that this comes from Lemma \ref{lem:contributions}.
\end{definition}

Now we arrive at the Main Theorem \ref{thm:main}.

\begin{theorem}
\label{thm:main}
A lower bound on the average genus of a 2-bridge knot with given crossing number $c$ is 
\begin{multline}
\frac{c-1}{2}-\left(\frac{3}{2(2^{c-2}+*)}\right)\left[\sum_{i=2}^{c-1}\sum_{d_1=0}^{i-2}\sum_{\substack{d_2=0 \\ c+d_1+d_2\equiv1}}^{c-i-1}\binom{i-2}{d_1}\delta(\{i\})\binom{c-i-1}{d_2} \right. \\ 
\left.+ \sum_{i=2}^{c-1}\sum_{d_1=0}^{i-2}\sum_{\substack{d_2=0 \\ c+d_1+1+d_2\equiv1}}^{c-i-1}\binom{i-2}{d_1}\delta(\{i,i+1\})\binom{c-i-1}{d_2}\right],
\end{multline}
where * is as in Theorem \ref{thm:number}.

This average is taken over all 2-bridge knots appearing twice except for those with palindromic type only appearing once as in Remark \ref{rem:palindrome}.
\end{theorem}

\begin{proof}
For convenience we refer to sum of the summations inside $\bm{[}$ $\bm{]}$ as ``$\bm{[}\text{\bf same as above}\bm{]}$.''

By the genus formula $g(K)=1-\frac{1+s-c}{2}$ for alternating knots we have left to show that the average number $s$ of Seifert circles must be at most $2+\left(\frac{3}{2^{c-2}+*}\right)\bm{[}\text{\bf same as above}\bm{]}$. 

By Theorem \ref{thm:number} on the number of knots on our list, we have left to show that the sum of all Seifert circles is at most $2\left(\frac{2^{c-2}+*}{3}\right)+\bm{[}\text{\bf same as above}\bm{]}$. 

By Corollary \ref{cor:upper} to Theorem \ref{thm:Seifert} on an upper bound for the number of Seifert cirles, we have left to show that the total number of vertically-smoothed crossings over all our 2-bridge knots of crossing number $c$ following Remark \ref{rem:palindrome} is $\bm{[}\text{\bf same as above}\bm{]}$. 

First observe that the first crossing $\sigma_1$ is always horizontally smoothed by Lemma \ref{lem:contributions} because it appears in position 1, and so we begin our sum with crossing index $i=2$.

Note next that no horizontally-smoothed crossings contribute to this sum based on the vertical indicator functions $\delta(\{i\})$ and $\delta(\{i,i+1\})$.  We have two sums:  the first corresponding to counting contributions from singles and the second to those from doubles.

Consider some vertically-smoothed crossing at crossing index $i$.  Its vertical indicator function $\delta$ must be 1.  We now count how many reduced billiard table words it appears in.

Suppose there are $d_1$ doubles prior to this crossing.  Then since the first run is a single $+$, there are only $i-2$ choices for where these doubles may occur, giving us $\binom{i-2}{d_1}$ different ways to write a reduced billiard table word $w$ up to the $i$th crossing.

If the $i$th crossing is a single, we perform the first summation; if it is double, we perform the second.  In the first case, we have the sum $c+d_1+d_2$ giving the reduced length $\ell$ of the word, which must be congruent to 1 modulo 3; in the second case we have $c+d_1+1+d_2$ accounting for the double at crossing index $i$.

There are now some number $d_2$ of doubles remaining in the $c-i-1$ positions (where the final crossing cannot have come from a double), giving us $\binom{c-i-1}{d_2}$ different ways to write a reduced billiard table word $w$ after the $i$th crossing.
\end{proof}

We note that there is a symmetry where index $i$ contributes the same as index $c+1-i$.

\begin{corollary}
\label{cor:sym}
For even $c$, the first summation need only go to $\frac{c}{2}$ with the total doubled.  For odd $c$ one may count doubly for $i$ up to $\lfloor \frac{c}{2} \rfloor$ and count singly $\lceil \frac{c}{2} \rceil$.
\end{corollary}

\begin{proof}
This is because the reduced billiard table word is of length 1 modulo 3 and because of symmetry of positions 2 and 3 in Proposition \ref{prop:billiardwrithe} and Lemma \ref{lem:contributions}.
\end{proof}

\begin{example}
\label{ex:c6ii}
We continue Example \ref{ex:c6} for the set of all knots obtained in this model for crossing number $c=6$ and demonstrate this computation via Main Theorem \ref{thm:main}.

Note that $c+d=6+d$ must be congruent to 1 modulo 3, so the number of doubles must be congruent to 1 modulo 3.

Because index $i=1$ never contributes to the number of Seifert circles, we begin with index $i=2$.  Only a single $-$ will contribute.  There are three more opportunities for doubles, but since $d\equiv 1\text{ mod } 3$, we have $\binom{3}{1}=3$ different knots with a vertically-smoothed crossing at index $i=2$.

At index $i=3$, we have one opportunity for a double prior and two opportunities for doubles afterward.  We may have a single $+$ here:  if there was a single $-$ immediately prior, this gives $\binom{1}{0}\binom{2}{1}=2$ different knots with a vertically-smoothed crossing $\sigma_1$ crossing at index $i=3$; if there was a double $-{}-$ immediately prior, the vertical indicator function is 0.  We may also have a double $++$, giving us $\binom{1}{0}\binom{2}{0}+\binom{1}{1}\binom{2}{2}=2$ knots with a vertically-smoothed $\sigma_2^{-1}$ crossing at index $i=3$.

By Corollary \ref{cor:sym}, this gives a total of $2(0+3+4)=14$ contributions to the number of Seifert circles, with 5 knots on the list, matching the numbers found in Example \ref{ex:c6}.
\end{example}

\begin{example}
\label{ex:c7ii}
We continue Example \ref{ex:c7} for the set of all knots obtained in this model for crossing number $c=7$.

Note that $c+d=7+d$ must be congruent to 1 modulo 3, so the number of doubles must be congruent to 0 modulo 3.

Because index $i=1$ never contributes to the number of Seifert circles, we begin with index $i=2$.  Only a single $-$ will contribute.  There are four more opportunities for doubles, and since $d\equiv 0\text{ mod } 3$, we have $\binom{4}{0}+\binom{4}{3}=5$ different knots with a vertically-smoothed crossing at index $i=2$.

At index $i=3$, we have one opportunity for a double prior and three opportunities for doubles afterward.  We may have a single $+$ here:  if there was a single $-$ immediately prior, this gives $\binom{1}{0}\binom{3}{0}+\binom{1}{0}\binom{3}{3}=2$ different knots with a vertically-smoothed crossing $\sigma_1$ crossing at index $i=3$; if there was a double $-{}-$ immediately prior, the vertical indicator function is 0.  We may also have a double $++$, giving us $\binom{1}{0}\binom{3}{2}+\binom{1}{1}\binom{3}{1}=6$ knots with a vertically-smoothed $\sigma_2^{-1}$ crossing here.

At index $i=4$, we have two opportunities for a double prior and two opportunities for doubles afterward.  We may have a single $-$ here, but then in order for the indicator function to be 1, the number of doubles prior must be $d_1\not\equiv 0\text{ mod }3$.  This gives $\binom{2}{1}\binom{2}{2}+\binom{2}{2}\binom{2}{1}=4$ different knots with a vertically-smoothed crossing $\sigma_2^{-1}$ crossing at index $i=4$.  We may have a double $-{}-$ here, but then $d_1\not\equiv 1\text{ mod }3$.  This gives $\binom{2}{0}\binom{2}{2}+\binom{2}{2}\binom{2}{0}=2$ different knots with a vertically-smoothed crossing $\sigma_1$ crossing here.

By Corollary \ref{cor:sym}, this gives a total of $2(0+5+8)+6=32$ contributions to the number of Seifert circles, with 11 knots on the list, matching the numbers found in Example \ref{ex:c7}.
\end{example}

\bibliographystyle{amsalpha}

\end{document}